\documentclass[12pt]{article}
\usepackage{amsmath,amsthm,verbatim,amssymb,amsfonts,amscd,diagrams, graphics, mathrsfs}
\theoremstyle{plain}
\newtheorem{theorem}{Theorem}[section]
\newtheorem{corollary}[theorem]{Corollary}
\newtheorem{lemma}[theorem]{Lemma}
\newtheorem{proposition}[theorem]{Proposition}

\theoremstyle{definition}
\newtheorem{definition}[theorem]{Definition}

\theoremstyle{remark}
\newtheorem{remark}[theorem]{Remark}

\newarrow{ul}---->
\newarrow{Backwards}<----   

\newcommand{\bve}{\bar \varepsilon}
\newcommand{\supp}{\text{supp}}
\newcommand{\onto}{\twoheadrightarrow}

\newcommand{\into}{\hookrightarrow}
\newcommand{\ms}{\mathscr}
\newcommand{\Z}{\mathbb{Z}}

\newcommand{\R}{\mathbb{R}}

\newcommand{\N}{\mathbb{N}}
\newcommand{\bd}{\partial}
\newcommand{\pf}{\pitchfork}

\renewcommand{\H}{\mathbb H}

\newcommand{\mc}[1]{\mathcal{#1}}

\newcommand{\dlim}{\varinjlim}

\newcommand{\mf}{\mathfrak}

\newcommand{\im}{\text{im}}
\newcommand{\cok}{\text{cok}}

\newcommand{\Mor}{\text{Mor}}

\begin{document}

\title{On the chain-level intersection pairing for PL pseudomanifolds}
\author{Greg Friedman}
\date{May 30, 2008}

\maketitle

\tableofcontents

\section{Introduction}

For a compact oriented PL manifold, $M$, the intersection pairing on chain complexes, which induces the intersection pairing algebra on $H_*(M)$, dates back to Lefschetz \cite{Lef}. However, Lefschetz's pairing does not provide an algebra structure on $C_*(M)$, itself, as two chains may only be intersected if they are in general position. This difficulty does not descend to the homology groups since any pair of cycles are homologous to cycles in general position, and the resulting intersection product turns out to be independent of the choices made while putting chains into general position. This approach to pairings and duality was supplanted eventually in manifold theory by the more versatile cup product algebra, but it  gained new relevance with the work of Goresky and MacPherson on intersection homology on PL pseudomanifolds (a class of spaces including complex varieties) in \cite{GM1} and is also related to the work
of Chas and Sullivan on string topology \cite{ChS} (see \cite{McC2}).

Nearly 80 years after Lefschetz, James McClure \cite{McC} has shown that the domain for the intersection pairing on $C_*(M)$ is in fact a \emph{full} subcomplex of $C_*(M)\otimes C_*(M)$. In other words, the subcomplex $G_2\subset C_*(M)\otimes C_*(M)$ on which the intersection product of chains \emph{is} well-defined is quasi-isomorphic to $C_*(M)\otimes C_*(M)$.\footnote{In order to be completely correct, this statement should incorporate some indexing shifts, which we leave out here in order not to clutter the introduction with too many technicalities; see  Section \ref{S: GP} for the correct statements. } In fact, McClure goes further to show that $C_*(M)$, together with the chain intersection pairing, has the structure of a \emph{partially-defined commutative DGA} and is thus quasi-isomorphic to an $E_{\infty}$ chain algebra. McClure's goal in doing so was to develop tools to study the Chas-Sullivan operations in string topology. 

With different purposes in mind, our first goal in this paper is to generalize McClure's result to the intersection pairing of intersection chain complexes on PL pseudomanifolds. In other words, we prove that  the domain $G_2$ of definition of the intersection pairing on an oriented PL pseudomanifold $X$ is a full subcomplex of $C_*(X)\otimes C_*(X)$, and, as a corollary, that the domain of the Goresky-MacPherson intersection pairing from $I^{\bar p}C_*(X)\otimes I^{\bar q}C_*(X)$ to $C_*(X)$ (or to $I^{\bar r}C_*(X)$, when $\bar r\geq \bar p+\bar q$) is a full subcomplex of $I^{\bar p}C_*(X)\otimes I^{\bar q}C_*(X)$. We then go on to show that the intersection pairing of intersection chains on a PL pseudomanifold possesses the structure of a \emph{partial restricted algebra}, in a sense to be made precise below.
 Although the first part of  this may seem to be a straightforward generalization of McClure's results (and we do, in fact, utilize McClure's superstructure), some of the details of McClure's proof use arguments that  rely strongly on the manifold structure of the spaces involved (in particular, McClure's proof rests on being able to cover his manifolds by Euclidean balls and then working with general position arguments within these balls) and thus  fail to work on stratified spaces. So, we must turn to alternate arguments that employ a generalization of McCrory's results on stratified general position \cite{Mc78} instead. Since PL manifolds are special cases of PL pseudomanifolds, our arguments include an alternative proof of McClure's theorem. 

This program is carried out in Sections \ref{S: GP}, \ref{S: pairing}, and \ref{S: Leinster}, below. In the first of these sections, we are concerned principally with general position issues and showing
that the (appropriately shifted) chain complex $C_*(X)\otimes \cdots \otimes C_*(X)$ contains a quasi-isomorphic subcomplex $G_k$ consisting of chains 
in stratified general position and whose boundaries are in stratified general position. We note here the important fact (in a certain sense the essence of the whole matter) that such chains \emph{cannot} in general be written as sums $C=\sum C_{i_1}\otimes \cdots \otimes C_{i_k}$ in which each collection  $C_{i_1},\ldots, C_{i_k}$ is in stratified stratified general position and has its boundary in stratified general position. This is completely analogous to the fact noted in \cite{McC} that a cycle $C$ might not be expressible as a sum of cycles of this form.  In general, there will be important canceling of boundary terms.
See below for a more technical, and hence more accurate, description, culminating in the statement of Theorem \ref{T: qi}. Similarly, we find a quasi-isomorphic subcomplex of the (appropriately shifted) tensor product of intersection chain complexes $I^{\bar p_1}C_*(X)\otimes \cdots \otimes I^{\bar p_k}C_*(X)$ that satisfies the appropriate stratified general position requirements; see Theorem \ref{T: pqi}.

In Section \ref{S: pairing}, we define an intersection chain multi-product, patterned after McClure's (which in turn relies on earlier prescriptions by Dold and others), whose domain is the  subcomplex of the tensor product of intersection chains that is constructed in Section \ref{S: GP}. We then show that this product restricts to the iteration of the PL intersection product of Goresky and MacPherson \cite{GM1} in the special case of $k$-tuples of chains whose tensor product lies in the domain (in general, not all chains in the domain can be written as sums of chains of this form).

Section \ref{S: Leinster} is concerned with the partial restricted algebra structure possessed by the intersection chain complexes. We will describe this idea more fully in a moment. 

In Section \ref{S: sheaves} (which is independent of Section \ref{S: Leinster}), as a first application of this circle of ideas, we demonstrate  that the sheaf theoretic intersection homology product defined by Goresky-MacPherson in \cite{GM2} (see also \cite[Section V.9]{Bo}) using abstract properties of the derived category of sheaves is equal on PL pseudomanifolds to that defined in \cite{GM1} using the geometric intersection pairing. While this result generally seems to be well-believed in the literature, we have not been able to pinpoint a prior proof. In addition, this approach has the benefit of providing  a very concrete ``roof of maps''  in the category of complexes of sheaves on $X$ that serves as a realization of the pairing morphism in $\Mor_{D^b(X)}(\mc I^{\bar p}\mc C^*\overset{L}{\otimes} \mc I^{\bar q}\mc C^*, \mc I^{\bar r}\mc C^*)$.

\paragraph{A categorical structure.} The material of Section \ref{S: Leinster}  places our results on domains of geometric intersection pairings into more categorical terms. This framework is due initially to Jim McClure and was refined by Mark Hovey. We provide some heuristics and motivations here; more precise details can be found below in Section \ref{S: Leinster}.

 Fix a number $n$ and consider classical perversities for
dimension $n$, i.e. functions $\bar p:\{2,3,\ldots,n\}\to \Z^+$ such that $\bar p(2)=0$ and $\bar p(j)\leq \bar p(j+1)\leq \bar p(j)+1$. Define $\bar p\leq \bar q$ if $\bar p(j)\leq \bar q(j)$ for all $j$.  This makes the set of
perversities into a poset, which we denote by $\mf P$.

By a {\it perverse chain complex}, we  mean a functor from $\mf P$ to the category 
of chain complexes.   An example of a perverse chain complex is the 
collection of intersection chain complexes $\{I^{\star}C_*(Y)\}$ for an $n$-dimensional stratified 
space $Y$, where $\star$ indexes the poset of perversities.

One then expects to define restricted chain algebras
that encompass product maps of the form 
\[
D_*^{\bar p} \otimes D_*^{\bar q} \to D_*^{\bar r}
\]
that are defined when $\bar p+\bar q\leq\bar r$, are compatible with the boundary maps,
and satisfy evident naturality, associativity, commutativity, and unital axioms.
The term ``restricted'' refers to the fact that the product is only defined
for pairs of perversities with $\bar p+\bar q$ less than or equal to the top
perversity.  

To accomplish this precisely requires a more formal setting, which has been worked out by Hovey \cite{Ho08}. 
One defines a symmetric monoidal product on the
category of perverse chain complexes by letting
\[
\{D_*^{\star}\}\boxtimes \{E_*^{\star}\}
\]
be the perverse chain complex that in perversity $\bar r$ is
\[
\dlim_{\bar p+\bar q\leq \bar r} D_*^{\bar p} \otimes E_*^{\bar q}.
\]
Then a restricted chain algebra should be a commutative monoid in the resulting symmetric monoidal category.

An example of a restricted chain algebra is that induced on the collection of 
shifted intersection homology groups $\{S^{-n}I^{\star}H_*(Y)\}$ for a stratified space $Y$, 
considered as chain complexes with zero differential, with the product defined by (direct limits of ) the Goresky-MacPherson intersection 
product defined in \cite{GM1}.

On the other hand, the collection $\{S^{-n}I^{\star}C_*(Y)\}$ of shifted intersection chain
complexes is {\it not} a restricted chain algebra because the chain-level
intersection pairing is only defined for pairs of chains that are in general position.

Let us say that a subobject of a perverse chain complex is {\it dense} (or \emph{full}) if the
inclusion map is a quasi-isomorphism for each $\bar p$.

By a {\it partial restricted chain algebra}, we mean a perverse chain complex 
$\{D_*^{\star}\}$ together with, for each $k$, a product defined on a dense 
subobject of $\{D_*^{\star}\}^{\boxtimes k}$; these partially-defined products
are required to have properties that are similar to the definition of 
commutative homotopy algebras in \cite{Lei} (see also \cite[Section 9]{McC}). We will define these objects more carefully below in Section \ref{S: Leinster} under the title of \emph{Leinster partial restricted commutative DGAs}.

In this language, our main theorem can be stated as follows. A more detailed explanation of the meaning of this theorem can be found in Section \ref{S: Leinster}. 

\begin{theorem}\label{T: PRCA}
For any compact oriented PL stratified pseudomanifold $Y$, the partially-defined intersection pairing on the perverse chain complex $\{S^{-n}I^{\star}C_*(Y)\}$ extends to the structure of  a
Leinster partial restricted commutative DGA.
\end{theorem}

Note that this partial chain algebra structure does not violate Steenrod's obstructions to the commutative cochain problem, since those obstructions apply only to everywhere-defined algebraic structures, not to partial algebraic structures.

\paragraph{ Future applications.}

James McClure, Scott O. Wilson, and the author are currently pursuing a program to demonstrate that the algebraic structures discovered here are homeomorphism invariants (at least over the rationals) in the following sense: the partial restricted algebras that correspond to homeomorphic pseudomanifolds are related by a chain of homomorphisms of partial restricted algebras that are weak equivalences, meaning that they induce isomorphisms at the level of homology. This is a stronger statement than that which follows from Goresky-MacPherson \cite{GM2}, which assures us only that there is a  homeomorphism-invariant restricted algebra structure in the derived category. 

Furthermore, Wilson's paper \cite{Wilson} implies that, over the
rationals, partial commutative DGAs  can be 
rectified to ordinary commutative DGAs. We propose to prove the analogous statement for  partial restricted chain
algebras. This would provide a way of assigning to a PL pseudomanifold a rational restricted commutative  DGA that could be seen as an ``intersection'' analogue of Sullivan's rational polynomial de Rham complex of PL forms. Such DGAs should prove interesting objects of study, perhaps leading
to a theory of   \emph{intersection rational homotopy groups}, or  to a singular space version of the Deligne-Griffiths-Morgan-Sullivan theorem \cite{DGMS}. These invariants would be more refined than classical rational homotopy theory in the same sense that intersection homology groups provide more refined information than ordinary homology groups on spaces carrying the appropriate filtration structures. Conjecturally, these may be a rational version of the intersection homotopy groups of Gajer \cite{Ga96, Ga96c}.

Further results over other coefficient rings may be possible by employing $E_\infty$ structures.

\paragraph{Acknowledgment.} I thank Jim McClure and Scott Wilson for many helpful discussions, and Jim McClure especially for providing motivation and straightening out the sign issues. Mark Hovey was instrumental in working out the details of the category of perverse chain complexes. 

\paragraph{A note on changes from the original version of \cite{McC}.}

During the initial writing of this paper, in particular the sections concerning the comparison of the Goresky-MacPherson intersection product with the generalized intersection pairing defined in Section \ref{S: multiproduct}, below, it became clear that certain signs (powers of $-1$) were not working out quite right. This led to a re-examination by McClure of his pairings in the original version of \cite{McC} and the discovery that some changes were necessary in order both to conform to Koszul sign conventions and to obtain the appropriate associativity of his multi-products. These changes will be described in a forthcoming revision of \cite{McC} by McClure and are incorporated into this paper. 
We provide here a short list of the main modifications as a convenience to the reader already familiar with the original version of \cite{McC} who would like a quick overview of what is different here.
The reasoning behind these changes, as well as the relevant definitions, are provided more fully as these notions arise, below; some of the more technical computations are collected in Appendix A, both for ease of access for those interested only in the changes from the original version of \cite{McC} and to avoid cluttering the main text even further than necessary. The  correct signs are due to McClure.

\begin{enumerate}
\item Our Poincar\'e duality map incorporates a sign $x\to (-1)^{m|x|}x\cap \Gamma$, where $\Gamma$ is the fundamental class of an $m$-dimensional oriented (pseudo-)manifold, and $|x|$ is the degree of the cohomology class $x$. See Section \ref{S: signs}.

\item We replace McClure's original exterior product $\varepsilon: C_*(X)\otimes C_*(Y)\to C_*(X\times Y)$ with a product $\bar \varepsilon: S^{-n_1}C_*(X)\otimes S^{-n_2}C_*(Y)\to S^{-n_1-n_2}C_*(X\times Y)$. This map is defined to be $(-1)^{\dim(X)\dim(Y)}$ times the composition of the appropriate (signed!) chain isomorphism $S^{-n_1}C_*(X)\otimes S^{-n_2}C_*(Y)\cong S^{-n_1-n_2}(C_*(X)\otimes C_*(Y))$ with $S^{-n_1-n_2}\varepsilon$. See Section \ref{S: prelims}.

\item $G_k$ is redefined in the obvious way to incorporate the shifts of the chain complexes involved, and the proofs of Theorems \ref{T: qi} and \ref{T: pqi}, corresponding to McClure's Proposition 12.3, must be modified to take these into account. In particular, some new care must be taken with the homotopy and product arguments.

\end{enumerate}

These sign issues are discussed further in Section \ref{S: signs}, throughout the text as they arise, and also in Appendix A, in which we verify some of the resulting fixes.

\section{Background}

In this section, we recall some background definitions.

\paragraph{Pseudomanifolds.} Let $c(Z)$ denote the open cone on the space $Z$,
and let $c(\emptyset)$ be a point. 

A \emph{stratified paracompact Hausdorff space}
$Y$ (see \cite{GM2} or \cite{CS}) is defined
by a filtration
\begin{equation*}
Y=Y^n\supset Y^{n-1} \supset Y^{n-2}\supset \cdots \supset Y^0\supset Y^{-1}=\emptyset
\end{equation*}
such that for each point $y\in Y_i=Y^i-Y^{i-1}$, there exists a \emph{distinguished neighborhood}
$U$ of $y$ such that there is a compact Hausdorff space $L$, a filtration  of $L$
\begin{equation*}
L=L^{n-i-1}\supset  \cdots \supset L^0\supset L^{-1}=\emptyset,
\end{equation*}
and a homeomorphism
\begin{equation*}
\phi: \R^i\times c(L)\to U
\end{equation*}
that takes $\R^i\times c(L^{j-1})$ onto $Y^{i+j}\cap U$. The subspace $Y_i=Y^i-Y^{i-1}$ is called the $i$th stratum, and, in particular, it is a (possibly empty) $i$-manifold. $L$ is called the \emph{link} of the component of the stratum; it is, in general, not uniquely determined, though it will be unique when $Y$ is a stratified PL pseudomanifold, as defined in the next paragraph.

A \emph{PL-pseudomanifold} of dimension $n$ is a PL space $X$ (equipped with a class of locally finite triangulations) containing a closed PL subspace $\Sigma$ of codimension at least 2 such that $X-\Sigma$ is a PL manifold of dimension $n$ dense in $X$. A \emph{stratified PL-pseudomanifold} of dimension $n$ is a PL pseudomanifold equipped with a specific filtration such that $\Sigma=X^{n-2}$ and the local normal triviality conditions of a stratified space hold with the trivializing homeomorphisms $\phi$ being PL homeomorphisms and each $L$ being, inductively, a PL pseudomanifold. In fact, for any PL-pseudomanifold $X$, such a stratification always exists such that the filtration refines
the standard filtration of $X$ by $k$-skeletons with respect to some triangulation \cite[Chptr. I]{Bo}. Furthermore, intersection homology is known to be a topological invariant of such spaces; in particular, it is invariant under choice of triangulation or stratification (see \cite{GM2}, \cite{Bo}, \cite{Ki}).

A PL pseudomanifold $X$ is oriented if $X-\Sigma$ is oriented as a manifold.

\paragraph{Intersection homology.}
In the context of PL-pseudomanifolds, the intersection chain complex, as defined initially by Goresky and MacPherson \cite{GM1} (see also \cite[Chapter I]{Bo}), is a subcomplex of the complex $C_*(X)$ of PL-chains on $X$. This $C_*(X)$ is a direct limit $\varinjlim_{T\in \mc T} C_*^T(X)$, where $C_*^T(X)$ is the simplicial chain complex with respect to the triangulation $T$ and  the direct limit is taken with respect to subdivision within a family of triangulations compatible with each other under subdivision and compatible with the stratification of $X$. 

Intersection chain complexes are subcomplexes of $C_*(X)$ defined with regard to  \emph{perversity parameters} $\bar p:\Z^{\geq 2}\to  \Z^+$ that are required to satisfy $\bar p(2)=0$ and $\bar p(k)\leq \bar p(k+1)\leq \bar p(k)+1$. We think of the perversity as taking the codimensions of the strata of $X$ as input. The output tells us the extent to which chains in the intersection chain complex will be allowed to intersect that stratum. Thus a simplex $\sigma$ in $C_i(X)$ (represented by a simplex in some triangulation) is deemed $\bar p$-\emph{allowable} if  $\dim(\sigma\cap X^{n-k})\leq i-k+\bar p(k)$, and a chain $\xi\in C_i(X)$ is $\bar p$-allowable if every simplex with non-zero coefficient in $\xi$ or $\bd \xi$  is allowable as a simplex. The allowable chains constitute the chain complex $I^{\bar p}C_*(X)$, and the $\bar p$-perversity intersection homology groups are the homology groups of this chain complex.

We also note here that one can proceed with two versions of this: one can use the usual compactly supported chains that, in a given triangulation, can be described by finitely many simplices with non-zero coefficients. Or, one may use Borel-Moore chains, for which one requires only that chains contain locally-finite numbers of simplices with non-zero coefficients. This latter case is important to the sheaf-theoretic version of intersection homology and will be important to us  in Section \ref{S: sheaves}, below. We will denote the Borel-Moore chain complex by $C^{\infty}_*(X)$, and, when we need to be precise, we will denote the compactly supported complex by $C^c_*(X)$. No decoration generally will imply compactly supported chains. The intersection chain complexes and homology groups will share the corresponding notation. 

For more background on intersection homology, we urge the reader to consult the exposition by Borel, et. al. \cite{Bo}. For both background and application of intersection homology in various fields of mathematics, the reader should see Kirwan and Woolf \cite{KIR}.

\section{Stratified general position for pseudomanifolds}\label{S: GP}

In this section, we study the domain for the intersection products of chains in a pseudomanifold. We begin by developing some preliminary notations and definitions based on those in McClure \cite{McC}. 

\subsection{Preliminaries and statements of theorems.}\label{S: prelims} Let $X$ be an $n$-dimensional PL stratified pseudomanifold. We will denote the $k$-fold product of $X$ with itself by $X(k)$ (to avoid confusion with the skeleton $X^m$). We give the product the obvious stratification: $X(k)^m=\bigcup_{\sum_i^k d_i=m}X^{d_1}\times \cdots \times X^{d_k}$. 

As in \cite{McC}, let $\bar k=\{1,\ldots, k\}$ for $k>1$, and let $\bar 0=\emptyset$. If $R:\bar k\to \bar k'$ is any map of sets, let $R^*: X(k')\to X(k)$ denote the induced composition $$X(k')=\text{Map}(\bar k',X)\to \text{Map}(\bar k,X)=X(k).$$ Then $R^*(x_1,\ldots, x_{k'})=(x_{R(1)},\ldots, x_{R(k)})$. These maps represent generalizations of the standard diagonal embedding $\Delta: X\into X\times X$.

We note that if $R:\bar k \to \bar k'$ is surjective, then $R^*: X(k')\to X(k)$ has the property that each component of each stratum of $X(k')$ injects into a component of a stratum of $X(k)$. In particular, the stratum $X_{d_1}\times \cdots \times X_{d_{k'}}$ injects into $X_{d_{R(1)}}\times \ldots \times X_{d_{R(k)}}$. Furthermore, for each stratum component of $X_{d_1}\times \cdots \times X_{d_{k}} \subset X(k)$, $(R^*)^{-1}(X_{d_1}\times \cdots \times X_{d_{k}})$ is either empty (if there exist $1\leq i,\ell\leq k$ such that $R(i)=R(\ell)$ but $d_i\neq d_{\ell}$) or contained in $X_{d_{R^{-1}(1)}}\times \ldots \times  X_{d_{R^{-1}(k')}}$ (if $d_i=d_{\ell}$ whenever $R(i)=R(\ell)$). Note that, in the latter case, each $d_{R^{-1}(a)}$ is well-defined precisely because of the condition that $d_i=d_{\ell}$ whenever $R(i)=R(\ell)$.

The following definition generalizes McClure's definition in \cite{McC} of general position for maps of manifolds:

\begin{definition} If $A$ is a PL subset of $X(k)$, we will say that $A$ is in \emph{stratified general position} with respect to $R^*$ if for each stratum component $Z=X_{d_1}\times \cdots \times X_{d_{k}}$ of $X(k)$ such that $d_i=d_\ell$ if $R(i)=R(\ell)$, we have 
\begin{equation}\label{E: gen pos}
\dim((R^*)^{-1}(A\cap Z))\leq \dim(A\cap Z)+ \sum_{i=1}^{k'}  d_{R^{-1}(i)}  - \sum_{i=1}^kd_i.
\end{equation} 
In other words, $A$ is in stratified general position with respect to $R^*$ if for each stratum component $Z$ of $X(k)$, $A\cap Z$ is in general position with respect to the map of manifolds from the stratum  containing $(R^*)^{-1}(Z)$ to $Z$.  A PL chain is said to be in stratified general position if its support is, and we write $C_*^{R^*}(X(k))$ for the subcomplex of PL chains $D$ of $C_*(X(k))$ such that both $D$ and $\bd D$ are in stratified general position with respect to $R^*$. 
\end{definition}

We will also need two other notions from \cite{McC}. First, for a differential graded complex $C_*$, we let $S^mC_*$ be the shifted complex with  $(S^mC_*)_i=C_{i-m}$ and $\bd_{S^mC_*}=(-1)^m\bd_{C_*}$. This last notation differs from \cite{McC}, where $\Sigma^m$ is used to denote the shift; we here reserve $\Sigma$ for singular loci of pseudomanifolds. This shift is introduced so that all maps, including the pairing maps to be introduced below, will be degree $0$ chain morphisms.  When $C_*(X)$ is a geometric chain complex, we let the notion of the support of a chain be independent of the functor; in other words, we take $|S^{-n}x|=|x|$, the geometric support of the  chain $x\in C_*(X)$.

\begin{remark} We note that for chain complexes $C_*$ and $D_*$, $S^{-m-n}(C_*\otimes D_*)$ and $S^{-m}C_*\otimes S^{-n}C_*$ are not in general isomorphic as chain complexes by the obvious homomorphism since  $\bd S^{-m-n}(c\otimes d) =(-1)^{m+n}S^{-m-n}\bd(c\otimes d)=(-1)^{m+n}S^{-m-n}(\bd c\otimes d+(-1)^{|c|}c\otimes \bd d)$, where $|c|$ is the degree of $c$. On the other hand, $\bd (S^{-m}c\otimes S^{-n}d)= (-1)^mS^{-m}\bd c\otimes S^{-n}d +(-1)^{m+|c|+n}S^{-n}c\otimes S^{-n} \bd d $. The appropriate isomorphism must take $S^{-m-n}(c\otimes d)$ to $(-1)^{n\deg(c)}S^{-m}c\otimes S^{-n}d$. This sign correction was not taken into account in the original version of \cite{McC}.

More generally, for complexes $A^i_*$, define $\Theta: S^{m_1}A^1_*\otimes \cdots \otimes S^{m_k}A^k_* \to S^{\sum m_i} (A_1^* \otimes \cdots \otimes A^k_*)$ by $$\Theta(S^{m_1}x_1\otimes \cdots \otimes S^{m_k}x_k)=(-1)^{\sum_{i=2}^k (m_i\sum_{j<i}|x_j|)} S^{\sum m_i}(x_1\otimes \cdots \times x_k).$$ This is a chain isomorphism; see Lemma \ref{L: Theta}, in Appendix A below.
\end{remark}

Secondly, we will need to consider the exterior product $\varepsilon$ defined in \cite[Section 7]{McC}.  The product $\varepsilon$ is the multilinear extension of the product that takes $\sigma_1\otimes \sigma_2$, where the $\sigma_i$ are oriented simplices, to a chain with support $|\sigma_1|\times |\sigma_2|$ and with appropriate orientation. This is a direct generalization of the standard simplicial cross product construction (see e.g. \cite{MK}); we refer the reader to \cite[Section 7]{McC} for details. The original version of \cite{McC} used only this product, but the revised version incorporates  a sign and grading correction in order to define $\bar \varepsilon$, which will be appropriately Poincar\'e dual to the cross product on cochains; without these sign and grading corrections, this duality occurs only up to signs.  In \cite{McC}, $\varepsilon_k$ is defined as a map from $C_*(M_1)\otimes \cdots \otimes  C_*(M_k)\to C_*(M_1\times \cdots \times M_k)$. With $\dim(X_i)=m_i$, we define $\bar \varepsilon_k: S^{-m_1}C_*(X_1)\otimes \cdots \otimes  S^{-m_k}C_*(X_k)\to S^{-\sum m_i}C_*(X_1\times \cdots \times  X_k)$ as $(-1)^{e_2(m_1,\ldots ,m_k)}$ times the composition of the chain isomorphism $\Theta: S^{-m_1}C_*(X_1)\otimes \cdots \otimes S^{-m_k}C_*(X_k)\to S^{-\sum m_i}(C_*(X_1 \times \cdots \times  X_k))$ described in the preceding paragraph with the $-\sum m_i$ shift of McClure's $\varepsilon$. Here $e_2(m_1,\ldots ,m_k)$ is the elementary symmetric polynomial of degree two on the symbols $m_1,\ldots ,m_k$, so $e_2(m_1,\ldots ,m_k)=\sum_{i=1}^k \sum_{j<i} m_im_j$.
In other words,  $\bar \varepsilon$ is the composite
\begin{diagram}
S^{-m_1}C_*(X_1)\otimes\cdots \otimes  S^{-m_k}C_*(X_k) &\rTo^{\Theta}& S^{-\sum m_i}(C_*(X_1)\otimes \cdots \otimes C_*(X_k))\\
{}&\rTo^{(-1)^{e_2}S^{-\sum m_i}\varepsilon}& S^{-\sum m_i}C_*(X_1\times \cdots \times  X_k).
\end{diagram}
As for $\varepsilon$, $\bar \varepsilon$ is a monomorphism. Furthermore, it is a degree $0$ chain map since $\Theta$ and $\varepsilon$ are. 

The map $\bar \varepsilon$ so-defined is Poincar\'e dual to the iterated cochain cross product; see Lemma \ref{L: eps dual}  in Appendix A. This version of the chain product also corrects the commutativity of Lemma 10.5b from the original version of \cite{McC}; see Lemma \ref{L: mccom} in Appendix A.

\subsubsection{Domains}  With the notation introduced above, we can define our domain for the intersection pairing:

\begin{definition}
For $k\geq 2$, let the \emph{domain} $G_k$ be the subcomplex of $(S^{-n}C_*X)^{\otimes k}$ consisting of elements $D$ such that both $\bar \varepsilon(D)$ and $\bar \varepsilon(\bd D)$ are in stratified general position with respect to all generalized diagonal maps, i.e. $$G_k=\bigcap_{k'<k}\bigcap_{R:\bar k\onto\bar k'} \bar \varepsilon^{-1}(S^{-nk}C_*^{R^*}(X(k))).$$
\end{definition}

\begin{remark}
The reason for the shifting is so that the intersection product becomes a degree $0$ chain map. See \cite{McC}.
\end{remark}

We can now state our main theorems concerning domains.

\begin{theorem}\label{T: qi}
The inclusion $G_k\into (S^{-n}C_*X)^{\otimes k}$ is a quasi-isomorphism for all $k\geq 1$.
\end{theorem}

For intersection chains, we must generalize slightly.

\begin{definition}
Let $P=(\bar p_1,\ldots,\bar p_k)$ be a collection of traditional perversities, and let $G_k^P=G_k\cap \left(S^{-n}IC_*^{\bar p_1}(X)\otimes \cdots\otimes S^{-n}IC_*^{\bar p_k}(X)\right)$. In other words, $G_k^P$ consists of those chains $D$ in $S^{-n}IC_*^{\bar p_1}(X)\otimes \cdots\otimes S^{-n}IC_*^{\bar p_k}(X)$ such that $\bar \varepsilon_k(D)$ and $\bar \varepsilon_k(\bd D)$ are in stratified general position with respect to $R^*$ for all surjective $R:\bar k\onto\bar k'$. 
\end{definition}

\begin{theorem}\label{T: pqi}
The inclusion $G_k^P\into S^{-n}IC_*^{\bar p_1}(X)\otimes \cdots\otimes S^{-n}IC_*^{\bar p_k}(X)$ is a quasi-isomorphism.
\end{theorem}

\begin{remark}\label{R: gen sup}
These theorems can be generalized to include other cases of interest in intersection homology. We could incorporate local coefficient systems defined on $X-\Sigma$, or, more generally, multiple local coefficient systems $\mf L_i$ and work with $S^{-n}IC_*^{\bar p_1}(X;\mf L_1)\otimes \cdots\otimes S^{-n}IC_*^{\bar p_k}(X;\mf L_k)$. We could also  instead consider the complexes $C_*^{\infty}(X)$ and $IC_*^{\infty}(X)$. In fact, the definitions of general position carry over immediately, and all homotopies constructed in the following proof are proper, thus they yield well-defined maps on these locally-finite chain complexes. The proofs that $G_k$ and $G_k^P$ are quasi-isomorphic to the appropriate tensor products is the same. We can also consider ``mixed type'' versions of $G_k$ that are quasi-isomorphic to $S^{-n}I^{\bar p_1}C_*^{\infty}(X)\otimes \cdots S^{-n}I^{\bar p_j}C_*^{\infty}(X)\otimes S^{-n}I^{\bar p_{j+1}}C_*^{c}(X)\otimes \cdots \otimes S^{-n}I^{\bar p_k}C_*^{c}(X)$. The necessary modifications are fairly direct.
\end{remark}

The proofs of Theorems \ref{T: qi} and \ref{T: pqi} (as well as of the more general cases mentioned in the remark) are nearly identical, so we will present only the proof of Theorem \ref{T: qi} in detail. For Theorem \ref{T: pqi}, we simply note that 
instead of chains of the form $\sum_A n_aS^{-n}\tau_{a_i}\otimes \cdots\otimes S^{-n}\tau_{a_k}$, we would  instead consider chains $\sum_A n_aS^{-n}\xi_{a_i}\otimes \cdots\otimes S^{-n}\xi_{a_k}$, where each $\xi_{a_j}$ is a $\bar p_j$ allowable chain in $X$. We note also that since all homotopies constructed below are stratum-preserving and proper, they preserve allowability of chains (compactly supported or not), and the induced homologies will also be allowable (see \cite{GBF3, GBF10}). 

\subsection{Proof of Theorem \ref{T: qi}}

In this subsection, we prove  Theorem \ref{T: qi}.

\begin{proof}[Proof of Theorem \ref{T: qi}.]
The proof follows the outline of that of McClure's \cite[Proposition 12.2]{McC}, and many of the steps are essentially identical. However, there are some points at which it is necessary to pay closer attention to the stratification, and one large step (our variant of McClure's Proposition 14.6) that must be done entirely differently. This is because McClure covers his manifolds with euclidean balls and then employs general position arguments within these euclidean structures; even modified versions of this covering approach seem to fail on stratified spaces. We will roughly follow the entire proof in order to achieve a sense of completeness and to ensure that all steps at which the stratification enters materially are properly addressed. However, we will refer often to \cite{McC}, particularly for steps that do not rely on explicit mention of the stratification. 

We recall that if $X$ is a stratified spaces, a (PL) homotopy  $H:Y\times I\to X$ is called \emph{stratum preserving} if for each $y\in Y$, $H(y,I)$ is contained in a single stratum of $X$. If $\phi:X\times I\to X$ is a stratum preserving homotopy and $l$ is an integer $1\leq l\leq k$, then there is an $l$th factor homotopy determined by $$X(k)\times I\cong X(l-1)\times X\times I\times X(k-l)\overset{\text{id}\times \phi\times \text{id} }{\to} X(k),$$ and  this homotopy is stratum-preserving. Our goal, generally speaking, is to use stratum-preserving homotopies to push chains into general position one factor at a time.

For this, we first need a version of 
McClure's Lemma 13.2 \cite{McC}, which says, essentially, that $l$th factor homotopies take products of chains to products of chains. It is fairly straightforward that McClure's lemma  remains true in our context, though we update the statement, mostly to take account of the change from $\varepsilon$ to $\bar \varepsilon$ (see above).

To account for some of the shifting, notice that $(S^{-nk}C_*(X(k)))\otimes C_*(I)$ is canonically isomorphic to $S^{-nk}(C_*(X(k))\otimes C_*(I))$ with no signs coming in (since we associate no suspension to the $C_*(I)$ term). Thus we have a well-defined chain map $S^{-nk}\varepsilon: S^{-nk}(C_*(X(k))\otimes C_*(I))\to S^{-nk}C_*(X(k)\times I)$.

\begin{lemma}\label{L: pres prod}
Let $h:X(k)\times I\to X(k)$ be an $l$th factor stratum-preserving homotopy, and suppose that $C$ is in the image of $\bar \varepsilon_k$. Then
\begin{enumerate}
\item $S^{-nk}(h\circ i_1)_* C$ is in the image of $\bar \varepsilon_k$, and
\item If $\iota$ is the canonical generator of $C_1(I)$, then $(S^{-nk}h_*)(S^{-nk}\varepsilon) (C\otimes \iota)$ is in the image of $\bar \varepsilon_k$.
\end{enumerate}
\end{lemma}

In other words, an $l$th factor homotopy homotopes  a product to a product, and the trace of the homotopy is also a product.

Next, we recall McClure's filtration \cite[Definition 13.3]{McC} of $(S^{-m}C_*M)^{\otimes k}$, modified here to take into account our shift conventions:
\begin{definition}
\begin{enumerate}
\item For $0\leq j\leq k$, define $\Lambda_j$ to be the set of all surjections $R:\bar k\onto \bar k'$ such that for each $i>j$, the set $R^{-1}(R(i))$ has only one element.

\item For $0\leq j\leq k$, let $G_k^j$ be the subcomplex of $(S^{-n}(C_*X))^{\otimes k}$ of chains $C$ for which $\bar \varepsilon_k(C)$ and $\bar \varepsilon_k(\bd C)$ are in stratified general position with respect to $R^*$ for all $R\in \Lambda_j$:
 $$G_k^j=\bigcap_{k'<k}\bigcap_{\overset{R:\bar k\to\bar k'}{R\in \Lambda_j}} \bar \varepsilon_k^{-1}(S^{-nk}C_*^{R^*}X(k)).$$
\end{enumerate}
\end{definition}

This yields the filtration
$$G_k=G_k^k\subset G_k^{k-1}\subset \cdots\subset C_k^0=(S^{-n}C_*X)^{\otimes k},$$ and we will prove the following proposition, which immediately implies Theorem \ref{T: qi}:

\begin{proposition}\label{P: filtered qi}
For each $j$, $1\leq j\leq k$, the inclusion $G_k^j\into G_k^{j-1}$ is a quasi-isomorphism.
\end{proposition}

The proof of this proposition relies on the following lemma, analogous to \cite[Lemma 13.5]{McC}.

\begin{lemma}\label{L: homotopy}
Suppose $D\in G_k^{j-1}$ and $\bd D\in G^j_k$. Then there is a $j$th factor  stratum-preserving homotopy $h:X(k)\times I\to X(k)$ such that 
\begin{enumerate}
\item $h\circ i_0$ is the identity,
\item the chains $S^{-nk}(h\circ i_1)_*(\bar \varepsilon_kD)$, $S^{-nk}(h\circ i_1)_*(\bar \varepsilon_k(\bd D))$, and $S^{-nk}h_*(S^{-nk}\varepsilon(\bar \varepsilon_k(\bd D)\otimes \iota))$ are in stratified general position with respect to $R^*$ for all $R\in \Lambda_j$,
\item $S^{-nk}h_*( S^{-nk} \varepsilon( \bar \varepsilon_k (D)\otimes \iota))$ is in stratified general position with respect to $R^*$ for all $R\in \Lambda_{j-1}$.
\end{enumerate}
\end{lemma}

Assuming this lemma, the proof of Proposition \ref{P: filtered qi} follows as for \cite[Proposition 13.4]{McC} by using the homotopy of the lemma to create the following homologies:
\begin{enumerate} 
\item for a cycle  $D\in G_k^{j-1}$, a homology  to a cycle $C\in G^j_k$ such that the homology is itself  in $G_k^{j-1}$, and 
\item for a chain $D\in G_k^{j-1}$ with $\bd D\in G_k^j$, a relative homology  to a chain in $G_k^j$ whose boundary is also $\bd D$, and such that the homology is itself  in $G_k^{j-1}$. 
\end{enumerate}
We run through the argument because the new shift conventions must be taken into account (though ultimately they don't do any harm to the essence of McClure's argument).

Given a cycle $D$ as stated and the $h$ guaranteed by the Lemma, $S^{-nk}(h\circ i_1)_* \bar \varepsilon_k D$ is in the image of $\bar \varepsilon_k$ by Lemma \ref{L: pres prod}. Since $\bar \varepsilon_k$ is a monomorphism, $C=(\bar\varepsilon_k)^{-1}(S^{-nk}(h\circ i_1)_* \bar \varepsilon_k D)$ is a well-defined cycle. By Lemma \ref{L: homotopy}, $C\in G_k^j$. 

Similarly,  by Lemma \ref{L: pres prod}, $(S^{-nk}h_*)(S^{-nk}\varepsilon) ((\bar \varepsilon_k D)\otimes \iota)$ is in the image of $\bar \varepsilon_k$. Let $E$ be the inverse image of this chain under $\bar\varepsilon_k$. By Lemma \ref{L: homotopy}, $E\in G_k^{j-1}$.

As in \cite{McC}, let $\lambda, \kappa\in C_0(I)$ such that $\bd \iota=\lambda-\kappa$. Then
\begin{align*}
\bve_k(\bd E)&=\bd (S^{-nk}h_*)(S^{-nk}\varepsilon) ((\bar \varepsilon_k D)\otimes \iota)\\
&=(S^{-nk}h_*)(S^{-nk}\varepsilon) \bd((\bar \varepsilon_k D)\otimes \iota)\\
&=(S^{-nk}h_*)(S^{-nk}\varepsilon) (\bd(\bar \varepsilon_k D)\otimes \iota+(-1)^{|\bar \varepsilon_k D|}\bve D\otimes (\lambda-\kappa))\\
&=0+(-1)^{|\bar \varepsilon_k D|}(S^{-nk}h_*)(S^{-nk}\varepsilon)(\bve_k D\otimes (\lambda-\kappa))\\
&=(-1)^{|\bar \varepsilon_k D|} (S^{-nk}(h\circ i_1)_*(\bve_k D)-S^{-nk}(h\circ i_0)_*(\bve_k D))\\
&= (-1)^{|\bar \varepsilon_k D|} \bve_k (C-D). 
\end{align*}
Thus $C$ and $D$ are homologous, since $\bve_k$ is a monomorphism. 

Similarly, to check the second statement, let $D\in G_k^{j-1}$ with $\bd D=C\in G_k^j$, and choose a homotopy $h$ as given by Lemma \ref{L: homotopy}. 

Then $S^{-nk}(h\circ i_1)_* \bve_kD$ and $(S^{-nk}h_*)(S^{-nk}\varepsilon) (\bve_k\bd D\otimes \iota)$ are in the image of $\bve_k$. Let $E_1$ and $E_2$ be the respective inverse images, which are in $G_k^{j}$ by Lemma \ref{L: homotopy}. Now

\begin{align*}
\bve_k(\bd E_2)&=\bd (S^{-nk}h_*)(S^{-nk}\varepsilon) (\bve_k\bd D\otimes \iota)\\
&=(-1)^{|\bve_k \bd D|} (S^{-nk}h_*)(S^{-nk}\varepsilon) (\bve_k\bd D\otimes (\lambda-\kappa))\\
&=(-1)^{|\bve_k \bd D|} (S^{-nk}(h\circ i_1)_*(\bve_k \bd D)-S^{-nk}(h\circ i_0)_*(\bve_k \bd D))\\
&=(-1)^{|\bve_k \bd D|}\bve_k (\bd E_1- C)\\
\end{align*} 
Thus, since $\bve_k$ is a monomorphism, $C=\bd (E_1+(-1)^{|\bve_k \bd D|}E_2)$.
\end{proof}

So we must prove Lemma \ref{L: homotopy}.

\begin{proof}[Proof of Lemma \ref{L: homotopy}]
To simplify the notation, we will assume that $j=k$. The other cases may be obtained by obvious modifications that would require overcomplicating the formulas that follow.

We suppose that $D$ is a chain in $G_k^{k-1}$ and that $\bd D\in  G^k_k$. We must show that there is a $k$th factor homotopy $h:X\times I\to X$ such that 
\begin{enumerate}
\item $h\circ i_0$ is the identity (where $i_0$ is the inclusion $X=X\times 0\into X\times I$),

\item $S^{-nk}(h\circ i_1)_*(\bar \varepsilon_kD)$, $S^{-nk}(h\circ i_1)_*(\bar \varepsilon_k(\bd D))$, and $S^{-n}h_*( S^{-nk}\varepsilon( \bve_k\bd D\otimes \iota))$ are in  stratified general position with respect to $R^*$ for all $R:\bar k\onto \bar k'$ (where $i_1$ is the inclusion $X=X\times 1\into X\times I$ and $\iota$ is the canonical chain in $C_1(I)$), and

\item $S^{-nk}h_*(S^{-nk} \varepsilon(\bve_k D\otimes \iota))$ is in stratified general position with respect to $R^*$ for all $R\in \Lambda_{k-1}$.
\end{enumerate}

We choose a triangulation $K$ of $X$ such that $D\in (c_*K)^{\otimes k}$. We let $\tau_1,\ldots, \tau_{\omega}$ be the simplices of $K$ with fixed, arbitrary orientations. Then we can write $$D=\sum_A n_A S^{-n}\tau_{a_1}\otimes \cdots\otimes S^{-n}\tau_{a_k},$$ where the sum runs over multi-indices $A=(a_1,\cdots, a_k)\in \{1,\ldots,\omega\}^k$. Similarly, $$\bd D=\sum_A n'_A S^{-n}\tau_{a_1}\otimes \cdots\otimes S^{-n}\tau_{a_k}.$$

To define the desired homotopy, we utilize the following proposition, which generalizes (and slightly strengthens) McClure's \cite[Proposition 14.6]{McC}. Although the proposition is analogous, it is the proof of this proposition for which we need most greatly differ from \cite{McC}, as we have not been able to construct a proof using McClure's methods.

\begin{proposition}\label{P: shift}
Let $X$ be a  stratified PL pseudomanifold of dimension $n$, and let $K$ be a triangulation of $X$. Then there is a stratum-preserving PL isotopy $\phi:X\times I\to X$ such that 
\begin{enumerate}
\item $\phi|_{X\times 0}$ is the identity,

\item if $\sigma$ and $\tau$ are simplices of $K$, then $\phi(\sigma,1)$ and $\tau$ are in stratified general position, i.e. $$\dim(\phi(\sigma,1)\cap \tau\cap X_\kappa)\leq \dim(\sigma\cap X_\kappa)+\dim(\tau\cap X_\kappa)-\kappa$$
for all $\kappa$, $0\leq \kappa\leq n$
(note that since $\phi$ is an isotopy, $\dim (\sigma\cap X_\kappa)=\dim(\phi(\sigma,1)\cap X_\kappa)$), and

\item \label{I: isotop} if $\sigma$ and $\tau$ are simplices of $K$, then for all $\kappa$, $0\leq \kappa\leq n$, $$\dim(\phi((\sigma\cap X_\kappa)\times I)\cap \tau)\leq \max(\dim(\supp(\phi((\sigma\cap X_\kappa)\times I)))+\dim(\tau\cap X_\kappa)-\kappa,\dim(\sigma\cap\tau\cap X_\kappa)).$$ 
\end{enumerate}
\end{proposition}

The proof of the proposition is deferred to below.

Let $h:X(k)\times I\to I$ be the $k$th factor homotopy obtained from the isotopy $\phi$ of the proposition. Note that since $\phi$ was stratum-preserving, so is $h$. Let $R:\bar k\to \bar k'$ be any surjection. We must verify that the various chains described above are in the appropriate general position with respect to $R^*$. We will show explicitly that $S^{-nk}h_*(S^{-nk}\varepsilon( \bve_k\bd D\otimes \iota))$ is in stratified general position with respect to $R^*$, the other proofs being similar. 

To simplify the notation somewhat in what follows, we will employ the following substitution in order to remove the shifts. Notice that, as far as supports of chains are concerned, the support of $S^{-nk}h_*(S^{-nk}\varepsilon( \bve_k\bd D\otimes \iota))=S^{-nk}h_*(S^{-nk}\varepsilon( \bve_k\sum_{A\mid n'_A\neq 0} n'_A S^{-n}\tau_{a_1}\otimes \cdots\otimes S^{-n}\tau_{a_k}\otimes \iota))$ is precisely the same as that of $h_*(\varepsilon_{k+1}( \sum_{A\mid n'_A\neq 0} n'_A S^{-n}\tau_{a_1}\otimes \cdots\otimes S^{-n}\tau_{a_k}\otimes \iota))\in C_*(X(k))$. This is because we are done taking boundaries at this point, so the various signs that come into play from the dimension shifts no longer need to be taken into account. The only thing that matters at this point are which terms are non-zero, and that is already settled. Thus for the purpose of checking the dimensions of intersections in order to make sure that stratified general position is satisfied (which is all that remains to do in this section), we are free to replace $\bd D$ with $\bd \bar D=\sum_{A\mid n'_A\neq 0} n'_A \tau_{a_1}\otimes \cdots\otimes \tau_{a_k}$ and to proceed using $\varepsilon$ and $h_*$ instead of their shifted versions. We make this change now.

We must consider what happens on each stratum $Z=X_{d_1}\times\cdots\times X_{d_k}$ of $X(k)$. As we have previously noted, if $R(a)=R(b)$ but $d_a\neq d_b$ for any pair $a,b\in \bar k$, then $R^*(X(k'))$ does not intersect this stratum and general position for this stratum is automatic. Therefore, we may confine ourselves to strata $X_{d_1}\times\cdots\times X_{d_k}$ of $X(k)$ for which $R(a)=R(b)$ implies $d_a=d_b$.

Now, $\supp(h_*(\varepsilon(\bd \bar D\otimes \iota)))$ is contained in the union over all $A$ such that $n'_A\neq 0$ of $$\tau_{a_1}\times \cdots \times \tau_{a_{k-1}}\times \phi(\tau_{a_k}\times I),$$ and so, letting $R^{-1}(R(k))=Q$, we see that $\supp(h_*(\varepsilon(\bd \bar D\otimes \iota)))\cap \im(R^*)$ is contained in the union over all $A$ such that $n'_A\neq 0$ of 
$$\left(\prod_{j\neq R(k)} \bigcap_{i\in R^{-1}(j)}\tau_{a_i}\right)\times \left(\phi(\tau_{a_k}\times I)\cap \bigcap_{i\in Q-\{k\}}\tau_{a_i}\right).$$
It follows that, on our stratum $Z$,
\begin{multline*}
\supp (h_*(\varepsilon(\bd \bar D\otimes \iota)))\cap \im(R^*)\cap Z\\
\subset \left(\prod_{j\neq R(k)} \bigcap_{i\in R^{-1}(j)}(\tau_{a_i}\cap X_{d_i})\right)\times \left((\phi(\tau_{a_k}\times I)\cap X_{d_k})\cap \bigcap_{i\in Q-\{k\}}(\tau_{a_i}\cap X_{d_i})\right).
\end{multline*}

Thus 
\begin{multline}\label{E: condition}
\dim(\supp(h_*(\varepsilon(\bd \bar D\otimes \iota)))\cap \im(R^*)\cap Z)\\
\leq\max_{n_A'\neq 0}\left(\sum_{j\neq R(k)} \dim \left(\bigcap_{i\in R^{-1}(j)}(\tau_{a_i}\cap X_{d_i})\right)\right. \\
\left.+ \dim\left((\phi(\tau_{a_k}\times I)\cap X_{d_k})\cap \bigcap_{i\in Q-\{k\}}\tau_{a_i}\cap X_{d_i}\right)\right).
\end{multline}

If $a_k$ is such that  $\dim(\phi(\tau_{a_k}\times I))<\dim(\tau_{a_k})+1$, then for any $A$ with $a_k$ as its final entry, we will have $h_*(\varepsilon(\tau_{a_1}\otimes\cdots\otimes \tau_{a_k}))=0$, since $h_*(\varepsilon(\tau_{a_1}\otimes\cdots\otimes \tau_{a_k}))$ must be a chain of dimension $1+\sum_{i=1}^k\dim(\tau_{a_i})$, while $\dim(h(\tau_{a_1}\times\cdots\times \tau_{a_k}))=\sum_{i=1}^{k-1}\dim(\tau_{a_i})+\dim(\phi(\tau_{a_k}\times I))$. In this case, $h_*(\varepsilon(\tau_{a_1}\otimes\cdots\otimes \tau_{a_k}))$ must trivially satisfy any general position requirements. So we may assume for the rest of the argument that $\dim(\phi(\tau_{a_k}\times I))=\dim(\tau_{a_k})+1$. 

By \eqref{E: gen pos}, it suffices to show for each remaining multi-index (those such that $n'_A\neq 0$ and $\dim(\phi(\tau_{a_k}\times I))=\dim(\tau_{a_k})+1$) that the righthand side of the inequality \eqref{E: condition} is
\begin{equation}\label{E: suffice}
\leq \dim(\supp(h_*(\varepsilon(\tau_{a_1}\otimes \cdots \tau_{a_k}\otimes \iota)))\cap Z)+\sum_{i=1}^{k'}d_{R^{-1}(i)}-\sum_{i=1}^kd_i.
\end{equation}
Note that $\supp(h_*\varepsilon(\bd \bar D\otimes \iota))\cap Z=\supp(h_*( (|\varepsilon_k\bd \bar D| \cap Z)\times I))$ since $h$ is stratum-preserving.

The following lemma will be used to complete the proof:

\begin{lemma}\label{L: stratum GP}
\begin{enumerate} 
\item For each $j\notin R(k)$, $$\dim \left(\bigcap_{i\in R^{-1}(j)}\tau_{a_i}\cap X_{d_{R^{-1}(j)}}\right)\leq (1-|R^{-1}(j)|)d_{R^{-1}(j)}+ \sum_{i\in R^{-1}(j)} \dim(\tau_{a_i}\cap X_{d_{R^{-1}(j)}})$$ 
(recall that $d_{R^{-1}(j)}$ is well-defined for the stratum $Z$)

\item 
\begin{multline*}\dim \left((\phi(\tau_{a_k}\times I)\cap X_{d_k})\cap \bigcap_{i\in Q-\{k\}}(\tau_{a_i}\cap X_{d_{R^{-1}(k)}})\right)\\
\leq (1-|Q|)d_{k}+ \dim(\supp(h_*(\varepsilon(\tau_{a_k}\otimes \iota)))\cap X_{d_k})\\\qquad +\sum_{i\in Q-\{k\}} \dim(\tau_{a_i}\cap X_{k})
\end{multline*}
\end{enumerate}
\end{lemma}

To see that this lemma suffices to finish the proof of Lemma \ref{L: homotopy}, we compute
\begin{align*}
\sum_{j\neq R(k)} &\dim \left(\bigcap_{i\in R^{-1}(j)}(\tau_{a_i}\cap X_{d_{R^{-1}(j)}})\right) + \dim\left((\phi(\tau_{a_k}\times I)\cap X_{d_k})\cap \bigcap_{i\in Q-\{k\}}(\tau_{a_i}\cap X_{d_k})\right)\\
&\leq \sum_{j\neq R(k)} \left((1-|R^{-1}(j)|)d_{R^{-1}(j)}+\sum_{i\in R^{-1}(j)} \dim(\tau_{a_i}\cap X_{d_{R^{-1}(j)}}) \right) +(1-|Q|)d_{k}\\&\qquad+\dim(\supp(h_*(\varepsilon(\tau_{a_k}\otimes \iota)))\cap X_{d_k}) + \sum_{i\in Q-\{k\}} \dim(\tau_{a_i}\cap X_{d_{R^{-1}(k)}}) \\
&= \sum_{j=1}^{k'}(1-|R^{-1}(j)|)d_{R^{-1}(j)}+\dim(\supp(h_*(\varepsilon(\tau_{a_k}\otimes \iota)))\cap X_{d_k})+ \sum_{i\neq k}\dim(\tau_{a_i}\cap X_{d_i})
\end{align*}

But $$\sum_{j=1}^{k'}|R^{-1}(j)|d_{R^{-1}(j)}=\dim Z=\sum_{i=1}^kd_i,$$ and 

$$\dim(\supp(h_*(\varepsilon(\tau_{a_k}\otimes \iota)))\cap X_{d_k})+ \sum_{i\neq k}\dim(\tau_{a_i}\cap X_{d_k})=\dim(\supp(h_*(\varepsilon(\tau_{a_1}\otimes \cdots \tau_{a_k}\otimes \iota)))\cap Z).$$
Thus the desired inequality \eqref{E: suffice} holds. 
\end{proof}

It remains to prove Proposition \ref{P: shift} and Lemma \ref{L: stratum GP}.

\begin{proof}[Proof of Lemma \ref{L: stratum GP}]
\begin{enumerate}
\item The proof is essentially the same as that of \cite[Lemma 14.5]{McC}: We continue to work with the stratum $Z$ and with a fixed $R$. 
Let $E=\tau_{a_1}\otimes \cdots\otimes \tau_{a_k}$.  Choose $j\in R^{-1}(k)$, and 
let $\bar R:\bar k\to \bar k''$ be any surjection that takes $R^{-1}(j)$ to $1$ and is bijective on $\bar k-R^{-1}(j)$.  Then $$\bigcap_{i\in R^{-1}(j)}(\tau_{a_i}\cap X_{d_{R^{-1}(j)}})=\bigcap_{i\in \bar R^{-1}(j)}(\tau_{a_i}\cap X_{d_{\bar R^{-1}(1)}}).$$ 

Now on the one hand, 
\begin{align*}
\dim&\left(\bigcap_{i\in \bar R^{-1}(1)}(\tau_{a_i}\cap X_{d_{R^{-1}(j)}})\times \prod_{i\notin \bar R^{-1}(1)} (\tau_{a_i}\cap X_{d_{R^{-1}(j)}})\right)\\
&=\dim\left(\bigcap_{i\in \bar R^{-1}(1)}(\tau_{a_i}\cap X_{d_{R^{-1}(j)}})\right)+\sum_{i\notin \bar R^{-1}(1)}\dim(\tau_{a_i}\cap X_{d_i}),
\end{align*}
while, on the other hand, since $\varepsilon_k(E)$ is in stratified general position with respect to any $R$ (by our standing assumptions),
\begin{align*}
\dim &\left(\bigcap_{i\in \bar R^{-1}(1)}(\tau_{a_i}\cap X_{d_{R^{-1}(j)}})\times \prod_{i\notin \bar R^{-1}(1)} (\tau_{a_i}\cap X_{d_i})\right)\\
&=\dim(\supp(\varepsilon_k(E))\cap \im(\bar R^*)\cap Z)\\
&\leq  \dim(\supp(\varepsilon_k(E))\cap Z)+ \sum_{i=1}^{k''}  d_{\bar R^{-1}(i)} -\sum_{i=1}^kd_i \qquad\qquad (\text{by stratified general position})\\
&= \dim(\supp(\varepsilon_k(E))\cap Z)+ d_{R^{-1}(j)}+\sum_{u\notin R^{-1}(j)} d_{u} -\sum_{i=1}^kd_i  \qquad\qquad(\text{by our choice of $\bar R$})\\
&=\dim(\supp(\varepsilon_k(E))\cap Z)+ d_{R^{-1}(j)} -\sum_{u\in R^{-1}(j)}d_u \\
&=\dim(\supp(\varepsilon_k(E))\cap Z)+ d_{R^{-1}(j)} (1- |R^{-1}(j)|)
\end{align*} 

Since $\dim(\supp(\varepsilon_k(E))\cap Z)=\sum_{i=1}^k\dim(\tau_{a_i}\cap X_{d_{i}})$, these two equations yield the result of the lemma.

\item By the same proof as in the first part of the lemma, 
$$\dim \left(\bigcap_{\overset{i\in Q}{i\neq k}}\tau_{a_i}\cap X_{d_{R^{-1}(k)}}\right)\leq (2-|Q|)d_{k}+ \sum_{i\in Q-\{k\}} \dim(\tau_{a_i}\cap X_{d_{k}}).$$ 
Now, by the conclusion of Proposition \ref{P: shift}, we can assume for any simplex $\eta$ in  $X_{d_{k}}$ (in particular for any simplex in $\bigcap_{i\in Q-\{k\}}(\tau_{a_i}\cap X_{d_{i}})$ that 
\begin{align*}
\dim&(\phi((\tau_{a_k}\cap X_{d_{k}})\times I))\cap \eta)\\
&\leq \max(\dim(\supp(\phi((\tau_{a_k}\cap X_{d_{k}})\times I)))+\dim(\eta\cap X_{d_{k}})-d_{k},\\
&\qquad\qquad\dim(\tau_{a_k}\cap\eta\cap X_{d_{k}})).
\end{align*}

If $\dim(\supp(\phi((\tau_{a_k}\cap X_{d_{k}})\times I)))+\dim(\eta\cap X_{d_{k}})-d_{k}$ is the larger number, then
\begin{align*}
\dim &\left((\phi(\tau_{a_k}\times I)\cap X_{d_k})\cap \bigcap_{i\in Q-\{k\}}(\tau_{a_i}\cap X_{d_{i}})\right)\\
&\leq 
\dim(\supp(\phi((\tau_{a_k}\cap X_{d_{R^{-1}(k)}})\times I)))
+ (2-|Q|)d_{k}\\
&\qquad+ \sum_{i\in R^{-1}(k)-\{k\}} (\dim(\tau_{a_i}\cap X_{d_{k}}))  -d_{k}\\
&\leq (1-|Q|)d_{k}+ \dim(\supp(h_*(\varepsilon(\tau_{a_k}\otimes \iota)))\cap X_{d_k})\\&\qquad +\sum_{i\in Q-\{k\}} \dim(\tau_{a_i}\cap X_{d_{k}})
\end{align*}
\end{enumerate}

If $\dim(\tau_{a_k}\cap\eta\cap X_{d_{k}})$ is the larger number, then

\begin{align*}
\dim &\left((\phi(\tau_{a_k}\times I)\cap X_{d_k})\cap \bigcap_{i\in Q-\{k\}}(\tau_{a_i}\cap X_{d_{k}})\right)\\
&\leq \dim\left(\tau_{a_k}\cap\bigcap_{i\in Q-\{k\}}(\tau_{a_i}\cap X_{d_{k}})\right)
\end{align*}
Once again, since $\bd \bar D$ was initially assumed to be in general position with respect to any $R$, it follows as in the proof of the first part of the lemma that this is 
$$\leq (1-|Q|)d_{k}+ \sum_{i\in Q} \dim(\tau_{a_i}\cap X_{d_{k}}),$$
which is certainly
\begin{align*}
&\leq (1-|Q|)d_{k}+ \dim(\supp(h_*(\varepsilon(\tau_{a_k}\otimes \iota)))\cap X_{d_k})\\&\qquad +\sum_{i\in Q-\{k\}} \dim(\tau_{a_i}\cap X_{d_{k}})
\end{align*}
\end{proof}

\begin{proof}[Proof of Proposition \ref{P: shift}]
The proof will make use of McCrory's proof of his stratified general position theorem \cite{Mc78}. This in turn makes use of the general position constructions for manifolds presented by Zeeman in \cite[Chapter 6]{Z}. McCrory shows that given a ``stratified polyhedron'' (the definition of which includes our stratified PL pseudomanifold $X$), and closed subpolyhedra $A,B,C$ such that $B\supset C$, then there exists an $\epsilon$-PL isotopy $H:X\times I\to X$ such that $H(c,t)=c$ for all $c\in C, t\in I$ and $H(B-C,1)$ and $A$ are in stratified general position, i.e. $(B-C)\cap X_i$ and $A\cap X_i$ are in general position in $X_i$ for each $i$-manifold $X_i$. 

The construction of McCrory's isotopy is by a double sequence of local ``$(j,\kappa)$-shifts'' that works up through the strata and down through the simplices of each stratum. In other words, one constructs a sequence of isotopies $G^1,\ldots, G^n$ such that $G^\kappa$ fixes $X^{\kappa-1}$ (and $C$), and each $G^\kappa$ is, in turn, composed of a sequence of isotopies $F^{j,\kappa}$, where the parameter $j$ descends through the dimensions of simplices\footnotemark of $X^\kappa-X^{\kappa-1}$.\footnotetext{By a simplex of $U-V$, we mean a simplex of $U$ that is not contained in $V$.} Each $F^{j,\kappa}$ consists of simultaneous disjoint local isotopies of neighborhoods in $X$ of the $j$-simplices of $X^{\kappa}-X^{\kappa-1}$. These neighborhoods may meet along their boundaries, but the local isotopies are fixed on these boundaries, so there is no problem with performing all of the local isotopies of $F^{j,\kappa}$ simultaneously. Each such $(j,\kappa)$-shift is constructed by applying in $X_\kappa$ Zeeman's shift construction for manifolds and then joining this Zeeman shift with the identity map on the link $L$ of the stratum (to be completely precise, one must also take into account the standard stratified homeomorphism between local neighborhoods of simplices and standard distinguished neighborhoods of the form $B^\kappa\times \bar cL$ - see \cite{Mc78}). By the arguments presented by McCrory \cite{Mc78} and Zeeman \cite{Z}, the end result of this sequence of isotopies puts $B-C$ in stratified general position with respect to $A$.

Our isotopy will be constructed similarly, by performing  $(j,\kappa)$-shifts for all possible simplices. For this purpose, there are two important points to note:

\begin{itemize}
\item As McCrory notes, the set $A$ comes into his construction only in that the triangulation $K$ is chosen so that $A$ is a subcomplex. Thus one should expect that the choice of $A$ is irrelevant beyond this, and thus by performing the appropriate $(j,\kappa)$-shifts, one can put any subcomplex of $K$ into general position with respect to all other subcomplexes of $K$ simultaneously.

In fact, it is not completely true that this is the only way that $A$ comes into the definition of McCrory's shift. In Zeeman's shift construction, which is the cornerstone of McCrory's, $A$ (there played by the symbol $Y$), also enters into which simplices have their neighborhoods shifted and into the definition of the shift. The first issue - only shifting simplices that actually intersect $A$ - simply limits the number of local isotopies being performed to avoid unnecessary ones, but performing extra local isotopies does no harm. As for how the shifts are actually defined, the only fact about $A$ that is significant in the definition of Zeeman's shift is that the intersection of $A$ with the link of the simplex $\eta$ whose neighborhood is to be shifted should not be the entire linking sphere (Zeeman works exclusively in the realm of manifolds). This allows one to construct a homeomorphism between the standard simplex (of the appropriate dimension) and the transverse disk to $\eta$ such that the intersection of $A$ with the link of $\eta$ gets mapped into a single face of the standard simplex. 

In our case (or McCrory's), if we wish to perform a $(j,\kappa)$-shift in a neighborhood of a $j$-simplex $\eta$ of $X_\kappa$ and the intersection of $A$ with the link of $\eta$ in $X_\kappa$ is the entire $\kappa-j-1$ dimensional linking sphere, then in fact an entire star neighborhood of $\eta$ in $X_\kappa$ will be contained in $A$, and general position with respect to $A\cap X_\kappa$ in $X_\kappa$ will automatically be satisfied. Thus in order to define each $(j,\kappa)$-shift for  each fixed $\kappa$, it will suffice to allow the $\kappa-1$ skeleton of $X^\kappa$ (in the relevant triangulation at the time) to play the role of $A$ for the purpose of applying Zeeman's construction to define the local shifts.  

\item  As for $A$, $B$ also comes into the construction of McCrory's isotopy both as determining which simplices should have their neighborhoods shifted and in the determination of the actual shifts. Once again, shifts in the McCrory construction are limited to those surrounding simplices that lie in $B$, but once again, this is an unnecessary limitation - making additional shifts does no harm. And once again, it is necessary for constructing Zeeman's shift within McCrory's that the intersection of $B$ (there called $X$) with the link in $X_\kappa$ of the simplex $\eta$ whose neighborhood we shift should not be the entire $\kappa-j-1$ sphere, and the reasons for this are identical. But once again, for the purpose of defining Zeeman's shift within the stratum $X_\kappa$, we may let the $\kappa-1$ skeleton of $X^\kappa$ play the role of Zeeman's $X$ (McCrory's $B\cap X_\kappa$), and then there is no difficulty defining the shift. Again, we are unconcerned with $\kappa$-simplices in $X_\kappa$ since these are automatically in general position with respect to any polyhedra in $X_\kappa$.

\end{itemize}

Thus, we construct a stratum-preserving isotopy $\phi_0:X\times I\to X$ as follows: For each $\kappa$, $1\leq \kappa\leq n$, we will define an isotopy $G^\kappa$ such that $G^\kappa|_{X^{\kappa-1}}$ is the identity. We will define $\phi_0$ to be the isotopy determined by performing the isotopies $G^1,\ldots, G^n$ successively.

Each $G^\kappa$ also comprises successive isotopies  $F^{\kappa-1,\kappa}\cdots F^{0,\kappa}$, and each $F^{j,\kappa}$ consists of performing McCrory's $(j,\kappa)$-shift for all $j$-simplices of  $X^\kappa-X^{\kappa-1}$
in the triangulation that has been arrived at to that point (by requiring each isotopy to be simplicial with respect to successive refinements of $K$). Each such $(j,\kappa)$-shift is built as in McCrory by joining the identity map of the link of the stratum with a Zeeman shift in $X_\kappa$ (utilizing, as in McCrory, the intermediate step of homeomorphing the appropriate local neighborhoods into standard distinguished neighborhoods). For the purpose of defining the Zeeman shift, we plug into Zeeman's machinery the intersection of the entire $\kappa-1$ skeleton of $X^\kappa$ (in the present triangulation) with $X_{\kappa-1}$. This skeleton plays the role of both McCrory's $A$ and $B$ (Zeeman's $X$ and $Y$). 

Now, suppose that $\sigma$ and $\tau$ are any simplices of the triangulation $K$ of $X$. We prove that $\phi_0$ takes $\sigma$ into stratified general position with respect to $\tau$. It suffices to show that $\phi_0$ takes $\sigma\cap X_\kappa$ into general position with respect to $\tau\cap X_\kappa$ for any $\kappa$, $0\leq \kappa\leq n$. This is trivial for $\kappa=0$. For $\kappa>0$, first the isotopies $G^1,\ldots, G^{\kappa-1}$ take $\sigma\cap X_\kappa$ to some subpolyhedron, say $Z$, of $X_\kappa$. $\tau\cap X_\kappa$ is also such a subpolyhedron. Suppose $\dim(Z\cap X_\kappa)$ (which is equal to $\dim(\sigma\cap X_\kappa)$) is equal to $\ell$ and $\dim(Z\cap \tau\cap X_\kappa)=s\leq \ell$. If $l=\kappa$, then we already have general position in this stratum. Otherwise, the isotopies $F^{\kappa-1,\kappa},\ldots, F^{\ell+1,\kappa}$ fix $Z$ since local shifts of neighborhoods of $t$-simplices fix the $t-1$ skeleton. Furthermore, while the shifts $F^{\ell, \kappa},\ldots,F^{s+1,\kappa}$ do isotop the $\ell,\ldots, s+1$ simplices of $Z$, since their interiors do not intersect $\tau$, it follows from the McCrory-Zeeman construction of the local shifts that the images of their interiors under the isotopy continue not to intersect $\tau$. From here, the movement of the present image of $Z$ under the further isotopies $F^{s,\kappa},\ldots, F^{\dim(\sigma\cap X_\kappa)+\dim(\tau\cap X_\kappa)-\kappa+1,\kappa}$ is exactly that of McCrory's isotopy, which pushes $Z$ into general position with respect to $\tau\cap X_\kappa$. Finally, any remaining isotopies $F^{\dim(\sigma\cap X_\kappa)+\dim(\tau\cap X_\kappa)-\kappa+1,\kappa},\ldots, F^{0,\kappa}$ do not damage this general position, by \cite[Lemma 30]{Z}. 

Since the further isotopies that constitute $G^{\kappa+1},\ldots, G^n$ fix $X^\kappa$, it follows that $\phi_0$ isotops $\sigma$ into stratified general position with respect to $\tau$.

To complete the proposition, we need to modify $\phi_0$ to an isotopy $\phi$ that also  satisfies condition \eqref{I: isotop} of the proposition. For this, let us consider $\phi_0: X\times I\to I$ as a PL map $\psi:X\times I\to X\times I$, given by $\psi(x,t)=(\phi_0(x,t),t)$. We may triangulate the domain and codomain copies of $X\times I$ so that the isotopy is simplicial and also so that each triangulation restricts on $X\times 0$ and $X\times 1$ to a refinement of $K$. We may also assume for each triangulation that $\eta\times I$ is a subcomplex for each $\eta$ in $K$. Now, taking the codomain copy of $X\times I$ with its triangulation, we construct a PL isotopy $\Phi: X\times I\times I\to X\times I$ just as we constructed $\phi_0$ above, but this time relative to $X\times 0$ and $X\times 1$ (in other words, $X\times 0$ and $X\times 1$ are held fixed). Such relative isotopies are also considered by McCrory and Zeeman, and we make the same modifications here as above - in particular we shift the neighborhoods of all possible simplices, this time except for those in $X\times 0$ and $X\times 1$. The previous arguments 
remain unchanged to demonstrate that for each $\sigma, \tau$ in $K$, $\Phi(\psi(\sigma\times I),1)\cap (X\times (0,1))$ is in stratified general position with respect to $\tau\times (0,1)$. In particular, 
\begin{multline*}
\dim(\Phi(\psi(\sigma\times I),1)\cap (\tau\times I)\cap (X_\kappa\times (0,1)))\\
\leq \dim(\sigma\cap X_\kappa) +1 +\dim(\tau\cap X_\kappa) +1-(\kappa+1)\\
=\dim(\sigma\cap X_\kappa)+1+\dim(\tau\cap X_\kappa)-\kappa.
\end{multline*}
Furthermore, of course,
$$\dim(\Phi(\psi(\sigma,1),1)\cap (\tau\times 1)\cap (X_\kappa\times1 ))=\dim(\phi_0(\sigma,1)\cap \tau\cap X_\kappa)\leq \dim(\sigma\cap X_\kappa) +\dim(\tau\cap X_\kappa) -\kappa$$
and
$$\dim(\Phi(\psi(\sigma,0),1)\cap (\tau\times 0)\cap (X_\kappa\times0 ))=\dim(\sigma\cap \tau \cap X_\kappa).$$

Now, let $\pi_X:X\times I\to X$ be the projection to $X$,  let $i_1:X\times I\to X\times I\times I$ be the inclusion into $X\times I\times 1$, and  define $\phi:X\times I\to X$ by $\phi=\pi_X  \Phi  i_1\psi$. To see that $\phi$ satisfies the desired requirements of the proposition, first notice that 
\begin{multline*}
\dim(\Phi((\sigma\cap X_\kappa)\times I)\cap \tau)\\
=\max(\dim(\phi((\sigma\cap X_\kappa)\times (0,1))\cap \tau),\dim(\phi(\sigma\cap X_\kappa,0)\cap \tau),\dim(\phi(\sigma\cap X_\kappa,1)\cap \tau)).
\end{multline*}
Now $\dim(\phi(\sigma\cap X_\kappa,0)\cap \tau)=\dim(\sigma\cap \tau\cap X_\kappa)$, and  $\dim(\phi(\sigma\cap X_\kappa, 1)\cap \tau)=\dim(\phi_0(\sigma\cap X_\kappa,1)\cap \tau)$.
Furthermore,  the 
the projection of $\Phi(\psi(\sigma\times I),1)\cap (\tau\times I)\cap (X_\kappa\times (0,1))$ to $X$ must contain $\phi((\sigma\cap X_\kappa)\times (0,1))\cap \tau$. By the stratified general position we have achieved, $\dim(\Phi(\psi(\sigma\times I),1)\cap (\tau\times I)\cap (X_\kappa\times (0,1)))\leq\dim(\sigma\cap X_\kappa)+1+\dim(\tau\cap X_\kappa)-\kappa$. If $\dim(\supp(\phi((\sigma\cap X_\kappa)\times (0,1))))=\dim(\sigma\cap X_\kappa)+1$, we are done. The only other possibility is that $\dim(\supp(\phi((\sigma\cap X_\kappa)\times (0,1))))=\dim(\sigma\cap X_\kappa)$. But by the definition of $\phi_0$, this is only possible  when $\dim(\sigma\cap X_{\kappa})=\kappa$. So $\phi((\sigma\cap X_\kappa)\times (0,1))\cap (\tau\times I)$ must have the form $(\sigma\cap X_{\kappa}\cap\tau)\times (0,1)$. In this case, the projection $X\times I\to X$ decreases the dimension of this intersection by $1$, and we still have $\dim(\Phi(\psi(\sigma\times I),1)\cap (\tau\times I)\cap (X_\kappa\times (0,1)))\leq
\dim(\supp(\phi((\sigma\cap X_\kappa)\times (0,1))))+\dim(\tau\cap X_\kappa)-\kappa$. 

This completes the proof.
\end{proof}

\section{Intersection pairings}\label{S: pairing}

In this section, we study the intersection pairings defined on our domains $G_k$ and $G_k^P$. In the simplest cases, these correspond to the simultaneous intersection of multiple chains, and we will indeed show that these intersection products correspond to the iteration of Goresky-MacPherson intersection products when such are defined (though not all generalized products have this form as the domains $G_k^P$ allow for the ``intersection'' of more general objects). Yet more general intersection products will arise in Section \ref{S: Leinster} within in the functorial machinery of the Leinster partial commutative algebra structure.

\subsection{Sign issues}\label{S: signs}

Our initial work on the following material was hampered by several difficulties that arose due to seeming inconsistencies in the signs (powers of $-1$) that occurred in the various formulas relating the general intersection pairing we define below with the iterated Goresky-MacPherson intersection pairing. Our struggle with these ``sign problems'' led back to an inconsistency with the Koszul sign conventions in the original version of McClure's paper \cite{McC}, and this problem was traced back to some sign issues involving the definition of the transfer map in Dold \cite{Dold}. Given a map of oriented manifolds $f:M\to M'$, Dold first defines his (homology) transfer maps $f_!:H_*(M')\to H_*(M)$ in the usual way as a composition of Poincar\'e duality on $M'$, followed by the cohomology pullback $f^*$, followed by Poincar\'e duality on $M$ (this is the gist of the construction - Dold actually considers quite general relative cases - see \cite[Section VIII.10]{Dold}). However, there is a cryptic note on page 314 of \cite{Dold}, noting that certain signs one should expect in resulting identities do not appear because the transfer should be defined ``in a more systematic treatment'' with a sign $(-1)^{(\dim M-j)(\dim M'-\dim M)}$, where $j$ is the dimension of the chain to which $f_!$ is being applied.

Applying this correction, however, did not completely fix the sign issues occurring here until it was noticed by McClure that the sign problem is not with the definition of the transfer but with the definition of Poincar\'e duality! McClure argues that the correct definition of the Poincar\'e duality map $P:H^*(M)\to H_{m-*}(M)$ for a closed oriented $m$-manifold $M$ with orientation class $\Gamma$ should be given by $P(x)=(-1)^{m|x|}(x\cap \Gamma)$. Note the sign. Of course this homomorphism is an isomorphism regardless of sign, but this should be considered the ``correct choice'' for the following reason:

Let $C_*(M)$ be the complex of (singular or simplicial) chains on $M$, and $C^*(M)$ the corresponding cochain complex. As usual, we can raise or lower indices and think of $C^*(M)$ as a complex with differential of degree $-1$ by setting $T_*(X)=C^{-*}(X)$. Then $H_*(T_*(X))=H^{-*}(X)$. Then we can think of $\cap \Gamma$ as a homomorphism $T_*(X)\to C_{*+m}(X)$, and $\cap \Gamma$ commutes with the differentials - see \cite[p. 243]{Dold}. But this is not the correct behavior for a map of degree $m$ according to the Koszul conventions! In order to be considered a chain map, a degree $m$ homomorphism should $(-1)^m$ commute with the differential - see \cite[Remark VI.10.5]{Dold}. If we instead use $P(x)=(-1)^{m|x|}(x\cap \Gamma)$, then we do obtain the desired $(-1)^m$ commutativity.

Note that this sign choice for Poincar\'e duality maps automatically incorporates the sign correction into Dold's transfer map since, letting $m=\dim M$ and $m'=\dim M'$, for $x\in C_j(M')$, we obtain $f_!(x)=  f^*((-1)^{(m'-j)m'} (\cap \Gamma_{M'})^{-1}(x))\cap \Gamma_M(-1)^{(m'-j)m} $, which is $(-1)^{(m'-j)(m'-m)}$ times Dold's transfer.

Furthermore, redefining $P(x)=(-1)^{m|x|}S^{-m}(x\cap \Gamma)$ makes this a degree $0$ chain map. 

We will use this convention for Poincar\'e duality throughout. Interestingly, this sign does not alter the sign of the Goresky-MacPherson intersection product \cite{GM1}; see the proof of Proposition \ref{P: GM}, below.

\subsection{An intersection homology multi-product}\label{S: multiproduct}\label{S: transfer}

In this section, we construct a generalized intersection product $\mu_k:G_k^P\to S^{-n}I^{\bar r}C_*(X)$, where $\bar r$ is a perversity greater than the sum of the perversities in $P$. This is done using a transfer (umkher) map that  is essentially a hybrid of the Poincar\'e-Whitehead duality utilized by Goresky-MacPherson \cite{GM1} and the umkehr map of McClure \cite{McC}.

We recall (see \cite[Section II.1]{Bo}) that if $A$ and $B$ are closed PL subspaces of respective dimensions $i$ and $i-1$ of a PL space $X$, then the chains $C\in C_i(X)$ that satisfy $|C|\subset A$ and $|\bd C|\subset B$ correspond bijectively to homology classes $[X]\in H_i(A,B)$. Thus ``in order to prescribe chains, we need only describe sets and homology classes.''

Now suppose $f:X^n\to Y^m$ is a PL map of compact oriented PL stratified  pseudomanifolds such that $f^{-1}(\Sigma_Y)\subset \Sigma_X$, where $\Sigma_X$ and  $\Sigma_Y$ are the respective singular sets of $X$ and $Y$. 
Suppose that $C\in C_i(Y)$ and $\dim(|C|\cap \Sigma_Y)<i$. Then $C$ corresponds to the homology class $[C]\in H_i(|C|,|\bd C|)$. Let $A=|C|$, $B=|\bd C|$, $A'=f^{-1}(A)$, and $B'=f^{-1}(B)$. We consider the following composition of maps
\begin{diagram}[LaTeXeqno] \label{E: transfer}
S^{-m} H_i(A, B)&\rTo & S^{-m}H_i(A\cup\Sigma_Y,B\cup\Sigma_Y)\\
{}&\rTo^{(-1)^{m(m-i)}(\cdot \cap \Gamma_Y)^{-1}}_{\cong} &H^{m-i}(Y-(B\cup\Sigma_Y),Y-(A\cup \Sigma_Y) )\\ 
{}&\rTo^{f^*}& H^{m-i}(X-(B'\cup\Sigma_X),X-(A'\cup \Sigma_X))\\
{}&\rTo^{\cap \Gamma_X(-1)^{n(m-i)}}_{\cong} &S^{-n}H_{i+n-m}(A'\cup \Sigma_X,B'\cup\Sigma_X)\\
\end{diagram}
The indicated signed cap products with the respective fundamental classes represent the Poincar\'e-Whitehead-Goresky-MacPherson duality isomorphism - see \cite[Appendix]{GM1}. We also incorporate the sign convention discussed above in Section \ref{S: signs}. 

Next, we note that $H_{i+n-m}(A'\cup \Sigma_X, B'\cup \Sigma_X)$ is isomorphic to $H_{i+n-m}(A', B'\cup (\Sigma_X\cap A'))$ by excising out $\Sigma_X-\Sigma_X\cap A'$. Furthermore, if $\dim(\Sigma_X\cap A')\leq i+n-m-2$, then by the long exact sequence of the triple (and another excision argument), $$H_{i+n-m}(A', B'\cup (\Sigma_X\cap A'))\cong H_{i+n-m}(A', B').$$ So, when this dimension condition is satisfied, we obtain a map 
 $S^{-m}H_i(A,B)\to S^{-n}H_{i+n-m}(A',B')$, which is a morphism of degree $0$.

\begin{remark}
N.B. It is the condition $\dim(\Sigma_X\cap f^{-1}(|C|))\leq i+n-m-2$ that will force the intersection pairing to be well-defined only for intersection chains in $G_k^P$ with certain perversity requirements on $P$. We cannot hope to have a well-defined pairing for $G_k$, in general, unless $X$ is, in fact, a manifold. 
\end{remark}

Now, let $\Delta$ be the diagonal map $X\into X(k)$. So if $R$ is the unique function $R:\bar k\to \bar 1$, then $\Delta=R^*$ in the notation of Section \ref{S: GP}. 

\begin{definition}\label{D: GPC}
Let $S^{-nk}C^{\Delta}_{*}(X(k))$ be the subcomplex of $S^{-nk}C_{*}(X(k))$ of chains $D$ such that if $D\in S^{-nk}C_{i+nk}(X(k))$ then
\begin{enumerate}
\item $\dim(|D|\cap \Sigma_{X(k)})<\dim|D|$,
\item $D$ is in stratified general position with respect to $\Delta$ (in particular, $\dim(\Delta^{-1}(|D|))\leq i+n$ and $\dim(\Delta^{-1}(|\bd  D|))\leq i+n-1)$, and
\item $\dim(\Delta^{-1}(|D|)\cap \Sigma_X)\leq i+n-2$, and $\dim(\Delta^{-1}(|\bd D|)\cap \Sigma_X)\leq i+n-3$
\end{enumerate}
\end{definition}

Then the chains in $S^{-nk}C_*^{\Delta}(X(k))$  satisfy all of the conditions outlined above for there to be  a well-defined degree $0$ chain homomorphism $\Delta_!:S^{-nk}C_*^{\Delta}(X(k))\to S^{-n}C_*(X)$ defined by taking the chain $D$ to  the homology class $[ D]\in S^{-nk}H_{i+nk}(| D|,|\bd D|)$ and then applying the composition
\begin{equation*} 
S^{-nk}H_{i+nk}(|D|,|\bd  D|)\to S^{-n}H_{i+n}(f^{-1}(|D|),f^{-1}(|\bd D|))\to S^{-n}C_{i+n}(X),
\end{equation*} 
where the first map is the composition described in diagram \eqref{E: transfer} and the second map makes use of the natural isomorphism between homology classes and chains recalled at the beginning of this subsection. 

\begin{definition}\label{D: trans shift} 
The morphism $\Delta_!:S^{-nk}C_*^{\Delta}(X(k))\to S^{-n}C_*(X)$ is a pseudomanifold version of a special case of the classical \emph{transfer} or \emph{umkehr} map. See \cite{Dold, McC} for more details. We show in Appendix A that $\Delta_!$ is indeed a chain map, and we also show there that the corresponding transfers $f_!: S^{-m}C^f_*(M)\to S^{-n}C_*(N)$ of \cite{McC} are chain maps, where $M^m,N^n$ are PL manifolds, $f$ is a PL map, and $C^f_*(M)$ is the chain complex of chains in general position with respect to $f$ - see \cite{McC}. 
\end{definition}

\begin{definition}
Suppose $P=\{\bar p_1,\ldots, \bar p_k\}$ is a sequence of traditional perversities and that $\bar p_1+\cdots + \bar p_k\leq \bar r$ for some traditional perversity $\bar r$. Then 
we let $\mu_k=\Delta_!\circ \bve_k: G_{k,*}^P\to S^{-n}C_*(X)$. Note that $\mu_1$ is the identity.
\end{definition}

We demonstrate in the following proposition that $\mu_k$ is well-defined  on appropriate $G_k^P$ and that its image lies in $S^{-n}I^{\bar r}C_*(X)$. 

\begin{proposition}\label{P: well-def prod}
Suppose $P=\{\bar p_1,\ldots, \bar p_k\}$ is a sequence of traditional perversities and that $\bar p_1+\cdots + \bar p_k\leq \bar r$ for some traditional perversity $\bar r$.
Then $\mu_k$ determines a well-defined chain map (of degree $0$) $G_k^P\to S^{-n}IC_*^{\bar r}(X)$.
\end{proposition}
\begin{proof}
Suppose $C\in (G_k^P)_i$. We must show that $\bve_k (C)\in S^{-nk}C^{\Delta}_*(X(k))$ so that $\mu_k$ is well-defined, and we must check that $\Delta_!\bve_k(C)$ is in $S^{-n}I^{\bar r}C_*(X)$. In particular, we must verify the three conditions of definition \ref{D: GPC}. But by definition of $G_k^P$, $\bve_k (C)$ is in general position with respect to $\Delta$, so the second condition is satisfied automatically. The first condition is also trivial since $|\xi|\cap \Sigma\subset |\bd \xi|\cap \Sigma$ for any intersection chain, and this implies the same for their products.

Now,  $C$ is represented by a chain in $ \oplus_{j_1+\cdots+j_k=i}(S^{-n}I^{\bar p_1}C_*(X))_{j_1}\otimes \cdots \otimes (S^{-n}I^{\bar p_k}C_*(X))_{j_k}$. Thus if $I$ is a $k$-component multi-index, $C$ breaks into a unique sum $\sum_{|I|=i} C_{I}$, where each $C_I$
lies in a  separate $(S^{-n}I^{\bar p_1}C_*(X))_{j_1}\otimes \cdots \otimes (S^{-n}I^{\bar p_k}C_*(X))_{j_k}$ with $\sum_{i=1}^k j_i=i$.   We know that  $\bar \varepsilon_k(C)\in S^{-nk}C_*(X(k))$, and  since $\Delta$ is the generalized diagonal, $\Delta^{-1}(|\bar \varepsilon_k C|)=|\bar \varepsilon_k(C)|\cap \Delta(X)$. Moreover, for each stratum $X_\kappa$, $\Delta^{-1}(X(k))\cap X_\kappa\cong \Delta(X_\kappa)\subset X_\kappa(k)$. 

Furthermore, for each multi-index $I=\{j_1,\ldots, j_k\}$, each $(S^{-n}I^{\bar p_1}C_*(X))_{j_1}\otimes \cdots \otimes (S^{-n}I^{\bar p_k}C_*(X))_{j_k}$  is generated by chains $S^{-n}\xi_{a_1}\otimes \cdots\otimes S^{-n}\xi_{a_k}$, where each $\xi_{a_\ell}$ is a $\bar p_l$ allowable chain. So for all $\ell$, $\dim(\xi_{a_\ell}\cap X_\kappa)\leq \dim(\xi_a)-(n-\kappa)+\bar p_\ell(n-\kappa)$. It follows that 
\begin{align*}
\dim(|\bar \varepsilon_k(S^{-n}\xi_{a_1}\otimes \cdots \otimes S^{-n}\xi_{a_k})|\cap X_\kappa(k))
&\leq \sum_{\ell=1}^k\dim(\xi_{a_\ell}\cap X_\kappa)\\
&\leq \sum_{\ell=1}^k\left(\dim(\xi_{a_\ell})-(n-\kappa)+\bar p_\ell(n-\kappa)\right)\\
&=i+nk-k(n-\kappa)+\sum_{\ell=1}^k\bar p_\ell(n-\kappa)\\
&=i+k\kappa+\sum_{\ell=1}^k\bar p_\ell(n-\kappa).
\end{align*}
This is true for all $S^{-n}\xi_{a_1}\otimes \cdots\otimes S^{-n}\xi_{a_k}\in (G_k^P)_i$, and since $C$, and hence each $C_I$, is in stratified general position with respect to $\Delta$, we have for each $\kappa$, $0\leq \kappa\leq n-2$, 
\begin{align*}
\dim(\Delta^{-1}(|\bar \varepsilon(C)|)\cap X_\kappa)
&=\dim(|\bar \varepsilon_kC|\cap \Delta(X_\kappa))\\
&\leq \dim(|\bar \varepsilon_kC|\cap X_{\kappa}(k))+\kappa-k\kappa \qquad \text{(by stratified general position)}\\
&\leq \max(\dim(|\bar \varepsilon_k(S^{-n}\xi_{a_1}\otimes \cdots \otimes S^{-n} \xi_{a_k})|\cap X_\kappa(k)))+\kappa-k\kappa\\
&\leq \left(i+k\kappa+\sum_{\ell=1}^k\bar p_\ell(n-\kappa)\right)+\kappa-k\kappa\\
&=i+\kappa+\sum_{\ell=1}^k\bar p_\ell(n-\kappa)\\
&\leq i+\kappa+\bar r(n-\kappa),
\end{align*}
where the maximum in the third line is over all $S^{-n}\xi_{a_1}\otimes \cdots \otimes S^{-n}\xi_{a_k}$ with non-zero coefficient in $C$. 
Since\footnote{Observe that it is critical here that $\bar r$ is a traditional perversity.}   $\bar r(n-\kappa)\leq n-\kappa-2$, it follows that $\dim(\Delta^{-1}(|\bar \varepsilon_kC_I|)\cap \Sigma_X)\leq i+n-2$. Thus $\dim(\Delta^{-1}(|\bve_kC|)\cap \Sigma_X)\leq i+n-2$. The same argument with $\bd C$ (which of course must be broken up into a different sum of tensor products of chains) shows that $\dim(\Delta^{-1}(|\bd \bve_kC|)\cap \Sigma_X)\leq i+n-3$. Thus $\bve_k C$ satisfies all the conditions of Definition \ref{D: GPC} and so lies in $S^{-nk}C^{\Delta}_*(X(k))$. It follows that 
$\mu_k(C)$ is well-defined in $S^{-n}C_*(X)$.

Moreover,  since $\bar \varepsilon(C)$ is an $i$-chain in $G^P_k$, the construction tells us that $\mu_k(C)$ will be an $i$-chain in $S^{-n}C_*(X)$, and thus it is represented by $S^{-n}\Xi$ for some $i+n$ chain $\Xi$. The preceding calculation shows that $\dim(|\Xi|\cap X_\kappa)=\dim(|\mu_k(C)|\cap X_\kappa)\leq i+\kappa+\bar r(n-\kappa)=(i+n)-(n-\kappa)+\bar r(n-\kappa)$, and thus $\Xi$ is $\bar r$-allowable. The same argument shows that $\bd \Xi$ is $\bar r$ allowable, so $\mu_k(C)\in (S^{-n}I^{\bar r}C_*(X))_i$. Note, however, that we cannot restrict the entire argument to primitives in the tensor product, as these might not lie in $G_k^P$; cancellation of boundary terms from different primitives is possible. Thus in considering $\bd C$, the maximum occurring in the last set of inequalities must occur over primitives that appear in $\bd C$ \emph{altogether}, not over boundary terms of individual primitives appearing in $C$. 

It is straightforward that $\mu_k$ is a chain map  since $\bar \varepsilon$ and $\Delta_!$ are and since the dimension conditions we have checked will hold for a sum of chains once they hold for each summand individually.
\end{proof}

\begin{corollary}\label{C: product}
Suppose $P=\{\bar p_1,\ldots, \bar p_k\}$ is a sequence of traditional perversities and that $\bar p_1+\cdots + \bar p_k\leq \bar r$ for some traditional perversity $\bar r$. Then there is a well-defined product (of degree 0) $\mu_{k*}:S^{-n}IH_*^{\bar p_1}(X)\times \cdots\times S^{-n}IH_*^{\bar p_k}(X)\to S^{-n} IH^{\bar r}(X)$. 
\end{corollary}
\begin{proof}
We can consider an element of $S^{-n}IH_*^{\bar p_1}(X)\times \cdots\times S^{-n}IH_*^{\bar p_k}(X)$ to be an element of $S^{-n}IH_*^{\bar p_1}(X)\otimes \cdots\otimes S^{-n}IH_*^{\bar p_k}(X)$. 
The corollary then follows from the proposition since $G_k^P$ is quasi-isomorphic to $S^{-n}IC_*^{\bar p_1}(X)\otimes \cdots\otimes S^{-n} IC_*^{\bar p_k}(X)$ by Theorem \ref{T: qi} and since $S^{-n}IH_*^{\bar p_1}(X)\otimes \cdots\otimes S^{-n}IH_*^{\bar p_k}(X)$ is a subgroup of $H_*(S^{-n}IC_*^{\bar p_1}(X)\otimes \cdots\otimes S^{-n}IC_*^{\bar p_k}(X))$ by the K\"unneth Theorem. 
\end{proof}

\begin{remark}
Since  $\Delta$ is a proper map, these considerations may be extended to noncompact oriented pseudomanifolds. In this case, if we continue to desire to study chains with compact supports, we simply replace the cohomology groups that occur in the above definition with the cohomology groups with compact supports, utilizing that version of Poincar\'e duality. There is no problem with the map $\Delta^*$ since $\Delta$ is proper. If we wish instead to consider locally-finite chains, we use the ordinary cohomology groups, but the Borel-Moore homology.\footnote{See \cite{Sp93} for an exposition of the relevant duality theorems. These theorems are stated there for manifolds, but we can adapt to the current situations by thickening the singular sets to their regular neighborhoods and employing some excision arguments and standard manifold doubling techniques.} Observe in this setting that if $|C|$ is not necessarily compact but $\Delta^{-1}(|C|)$ is, then $H^c_{i}(\Delta^{-1}(|C|)\cup \Sigma_X,\Delta^{-1}(|\bd C|)\cup\Sigma_X)\cong H_{i}^{\infty}(\Delta^{-1}(|C|)\cup \Sigma_X,\Delta^{-1}(|\bd C|)\cup\Sigma_X)$, as follows from an excision argument. We also note that, in the case of locally-finite chains, we can functorially restrict to open subsets $U$ of $X$ to get a map
$$S^{-nk}H^{\infty}_*(|C|\cap U(k),|\bd C|\cap U(k))\to S^{-n}H^{\infty}_{*}(\Delta^{-1}(|C|)\cap U,\Delta^{-1}(|\bd C|)\cap U).$$
\end{remark}

\begin{remark}
The transfer map discussed here can also be generalized to appropriate stratified maps $f:X\to Y$ between oriented stratified PL pseudomanifolds in order to obtain a transfer $f_!$ from subcomplexes of intersection chain complexes of $Y$ satisfying appropriate stratified general position conditions to intersection chain complexes of $X$. Since we do not need such generality here, we do not investigate the relevant details. 
\end{remark}

\subsection{Comparison with Goresky-MacPherson product}\label{S: GM}

In this section, we study the compatibility between the intersection product $\mu_k$ and the Goresky-MacPherson intersection product of \cite{GM1} for those instances when our element of $G_k^P$ can be written as a product of chains in stratified general position. Recall that we have introduced a sign in the Poincar\'e-Whitehead-Goresky-MacPherson duality map; see Section \ref{S: signs}. We first consider the case $k=2$ and then generalize to more terms. This will require us to demonstrate that iteration of the Goresky-MacPherson product is well-defined. 

We first show that, when $k=2$, our product is the Goresky-MacPherson intersection product, in those cases where the Goresky-MacPherson product is defined, in particular for two chain in appropriate stratified general position \cite{GM1}. In order to avoid confusion with the cap product, we denote the Goresky-MacPherson pairing by $\pf$, though this symbol is used for a somewhat different, but related, purpose in \cite{GM1}.

\begin{proposition}\label{P: GM}
Suppose that $C\in IC_i^{\bar p}(X)$ and $D\in I^{\bar q}C_j(X)$ are two chains in stratified general position, that $\bd C$ and $D$ are in stratified general position, and that $C$ and $\bd D$ are in stratified general position. Suppose $\bar p+\bar q\leq \bar r$, where $\bar r$ is also a traditional perversity. 
Then $S^n\mu_2(S^{-n}C\otimes S^{-n}D)=C\pf D$. 
\end{proposition}

\begin{proof}[Proof of Proposition \ref{P: GM}]
Let $C$ and $D$ be the indicated chains. We note that two chains being in stratified general position is the same thing as their product under $\bar \varepsilon$ being in stratified general position with respect to $\Delta:X\to X\times X$. We trace through the definitions.  

Recall the definition of the Goresky-MacPherson product: $C\times D$ represents an element of $H_i(|C|,|\bd C|)\times H_j(|D|,|\bd D|)$, which is taken to an element, represented by the same pair of chains, of $H_i(|C|\cup J,J)\times H_j(|D|\cup J, J)$, where $J=|\bd C|\cup|\bd D|\cup \Sigma_X$. Next one applies the inverse to the Poincar\'e-Whitehead-Goresky-MacPherson duality isomorphism represented by the (signed!)  inverse to the cap product with the fundamental class. Let $\Gamma$ denote the fundamental class of $X$, and let $\Upsilon=(\cap \Gamma)^{-1}$, which acts on the right as for cap products. For a constructible pair $(B,A)\subset X$ with $B-A\subset X-\Sigma$, $\Upsilon$ is a well-defined isomorphism $H_{i}(X-A,X-B)\to H^{n-i}(B,A)$; see \cite[Section 7]{GM1}. The Goresky-MacPherson product is represented, up to excisions, by the chain $$(((-1)^{n(n-i)}[C]\Upsilon)\cup ( (-1)^{n(n-j)}[D]\Upsilon))\cap \Gamma (-1)^{n(n-i+n-j)}=(([C]\Upsilon)\cup ( [D]\Upsilon))\cap \Gamma$$ 
in $H_{i+j-n}(|C|\cap|D|,(|\bd C|\cap|D|)\cup (|C|\cap|\bd D|))$ (see \cite[Section 2.1]{GM1}). 
Note that the sign we have introduced in the Poincar\'e duality map does not affect the sign of the Goresky-MacPherson product $\pf$. 

Let $\Upsilon_{2}$ denote the inverse of $\cap (\Gamma\times \Gamma)$, which induces the Poincar\'e-Whitehead-Goresky-MacPherson duality isomorphisms on the pseudomanifold $X\otimes X$. The image of $S^{-n}C\times S^{-n} D$ under $\mu_2$, as defined above, is represented by \begin{align*}
S^{-n}&(\Delta^*((\bar \varepsilon(S^{-n}C\otimes S^{-n}D))\Upsilon_2(-1)^{2n(2n-i-j)})\cap \Gamma (-1)^{n(2n-i-j)}\\
&=S^{-n}(\Delta^*((\bar \varepsilon(S^{-n}C\otimes S^{-n}D))\Upsilon_2))\cap \Gamma (-1)^{n(-i-j)}\\
&=(-1)^{n(-i-j)} S^{-n} (\Delta^*(  (-1)^{n^2+ni} S^{-n}(C\times D))\Upsilon_2))\cap \Gamma \\
&=(-1)^{n(n-j)} S^{-n} (\Delta^*( (S^{-n}(C\times D))\Upsilon_2))\cap \Gamma.
\end{align*}
The second equality comes from the definition of $\bve$.

In order to make the comparison with the Goresky-MacPherson product more precise, notice that, by excision isomorphisms, we can also describe $\mu_2$,  by $(-1)^{n(n-j)}$ times the composition

\begin{align*}
H_{i+j}(|C\times D|,|\bd(C\times D)|)&\to H_{i+j}(|C\times D|,|\bd (C\times D)\cap ((J\times X)\cup (X\times J)))\\
&\cong H_{i+j}(|C\times D|\cup((J\times X)\cup (X\times J)),  ((J\times X)\cup (X\times J)))\\
&\overset{\Upsilon_2}{\cong}H^{2n-i-j}(X\times X- ((J\times X)\cup (X\times J)),\\&\qquad\qquad\qquad\qquad X\times X-|C\times D|\cup ((J\times X)\cup (X\times J)))\\
&\overset{\Delta^*}{\to}H^{2n-i-j}(X-J,X-|C\cap D|\cup J)\\
&\overset{\cap \Gamma}{\cong}H_{i+j-n}(|C\cap D|\cup J,J)\\
&\cong H_{i+j-n}(|C\cap D|,|(|\bd C|\cap |D|)\cup (|C|\cap|\bd D|)),
\end{align*}
followed by the shift to put the associated chain in $S^{-n}C_{i+j}(X)$.

Ignoring the shifts, which we may do at this point without disrupting any signs, it therefore suffices to compare $\Delta^*([C\times D]\Upsilon_2)$ with $([C]\Upsilon)\cup ([D]\Upsilon )$ in $H^{2n-i-j}(X-J,X-|C\cap D|\cup J)$. The usual formula for the cup product says that the latter is equal to $\Delta^*([C]\Upsilon\times  [D]\Upsilon)$, where this $\times$ denotes the cochain cross product. So we compare $[C\times D]\Upsilon_{2}$ with $[C]\Upsilon\times [D]\Upsilon$ in $H^{2n-i-j}(X\times X- ((J\times X)\cup (X\times J)), X\times X-|C\times D|\cup ((J\times X)\cup (X\times J)))$. Taking the cap product with $\Gamma \times \Gamma$ of the former gives $[C\times D]=[C]\times [D]\in  H_{i+j}(|C\times D|\cup((J\times X)\cup (X\times J)),  ((J\times X)\cup (X\times J)))$ (which corresponds to the homology cross product of $C$ and $D$), while taking this cap product with $[C]\Upsilon\times [D]\Upsilon$ gives $(-1)^{n(n-j)}([C]\Upsilon\cap \Gamma) \times ([D]\Upsilon\cap \Gamma)=
(-1)^{n(n-j)}[C] \times [D]\in  H_{i+j}(|C\times D|\cup((J\times X)\cup (X\times J)),  ((J\times X)\cup (X\times J)))$ (see \cite[Theorem 5.4]{BRTG} ).

Thus the sign $(-1)^{n(n-j)}$ appears twice, so they cancel, completing the proof.
\end{proof}

\begin{remark}\label{R: fudging}
In the computations that follow, for the sake of simplicity of notation, we suppress the excisions and allow appropriate chains and cochains to stand for the elements of the respective homology and cohomology groups such as those considered in the preceding proof. Each computation could be performed in more detail by modeling the above arguments more closely.
\end{remark}

\begin{corollary}\label{C: iteration}
Suppose $P=\{\bar p_1,\ldots, \bar p_k\}$ is a sequence of traditional perversities and that $\bar p_1+\cdots + \bar p_k\leq \bar r$ for some traditional perversity $\bar r$. Then if $D_i\in I^{\bar p_i}C_*(X)$ and $(\otimes_{i=1}^kS^{-n}D_i)\in G_k^P$, the product $S^{n}\mu_k(\otimes_{i=1}^kS^{-n}D_i)\in IC_*^{\bar r}(X)$ is equal  to  the iterated Goresky-MacPherson intersection product of the chains $D_i$.
\end{corollary}

Before proving the corollary, we must first demonstrate that iterating the Goresky-MacPherson intersection pairing is even possible in consideration of the necessary perversity compatibilities. This is the goal of the following lemmas.

\begin{definition}\label{D: n-perv}
Let an $n$-perversity be a (traditional Goresky-MacPherson) perversity whose domain is restricted to integers $2\leq \kappa\leq n$. 
\end{definition}

\begin{lemma}\label{L: min perv}
Let $\bar p$ and $\bar q$ be two $n$-perversities such that there exists an $n$-perversity $\bar r$ with $\bar p(\kappa)+\bar q(\kappa)\leq \bar r(\kappa)$ for all $2\leq \kappa\leq n$. There there exists a unique minimal perversity $\bar s$ such $\bar p(\kappa)+\bar q(\kappa)\leq \bar s(\kappa)$ for all $2\leq \kappa\leq n$. (By minimal, we mean that for any $\bar r$ such that  $\bar p(\kappa)+\bar q(\kappa)\leq \bar r(\kappa)$ for all $2\leq \kappa\leq n$, $\bar r(\kappa)\geq \bar s(\kappa)$.)  
\end{lemma}
\begin{proof}
We construct $\bar s$ inductively as follows: Let $\bar s(n)=\bar p(n)+\bar q(n)$. For each $\kappa<n$ (working backwards from $n-1$ to $2$): if $\bar p(\kappa)+\bar q(\kappa)<\bar s(\kappa+1)$, let $\bar s(\kappa)=\bar s(\kappa+1)-1$; and if $\bar p(\kappa)+\bar q(\kappa)=\bar s(\kappa+1)$, let $\bar s(\kappa)=\bar s(\kappa+1)$. We note that by construction we must always have $\bar s(\kappa)\geq \bar p(\kappa)+\bar q(\kappa)$, and it is clear that $\bar s$ is minimal with respect to this property among all functions $\bar f$  satisfying $\bar f(\kappa)\leq \bar f(\kappa+1)\leq \bar f(\kappa)+1$ (for all $\kappa$, $\bar s(\kappa)$ is as low as possible to still be able to ``clear the jumps''). $\bar s$ is certainly a perversity, provided that $\bar s(2)=0$, but this must be the case since we know that $\bar p+\bar q\leq \bar s\leq \bar r$, and $\bar r(2)=\bar p(2)=\bar q(2)=0$.  
\end{proof}

\begin{definition}
Given the situation of the preceding lemma, we will call $\bar s$ the \emph{minimal $n$-perversity over $\bar p$ and $\bar q$}.
\end{definition}

\begin{lemma}\label{L: perverse sum}
Let $\bar p$ and $\bar q$ be two $n$-perversities and let $\bar f:\{2,\ldots, n\}\to \N$ be a non-decreasing function such that $\bar p(\kappa)+\bar q(\kappa)+\bar f(\kappa)\leq \bar r(\kappa)$ for some $n$-perversity $\bar r$ and for all $2\leq \kappa\leq n$. Let $\bar s(\kappa)$ be the minimal $n$-perversity over $\bar p$ and $\bar q$. Then $\bar s(\kappa)+\bar f(\kappa)\leq \bar r(\kappa)$. 
\end{lemma}
\begin{proof}
Since $\bar s(n)=\bar p(n)+\bar q(n)$ (see the proof of Lemma \ref{L: min perv}), we have $\bar s(n)+\bar f(n)\leq \bar r(n)$. Suppose now that $\bar s(\kappa+1)+\bar f(\kappa+1)\leq \bar r(\kappa+1)$ for some $\kappa$, $2\leq \kappa\leq n-1$.   If $\bar p(\kappa)+\bar q(\kappa)<\bar s(\kappa+1)$, then $\bar s(\kappa)=\bar s(\kappa+1)-1$, and we must have $\bar s(\kappa)+\bar f(\kappa)\leq \bar r(\kappa)$ since $\bar r(\kappa)\geq \bar r(\kappa+1)-1$. If $\bar p(\kappa)+\bar q(\kappa)=\bar s(\kappa+1)$, then $\bar s(\kappa)=\bar s(\kappa+1)=\bar p(\kappa)+\bar q(\kappa)$, and so again $\bar s(\kappa)+\bar f(\kappa)\leq \bar r(\kappa)$, this time by hypothesis. The proof is complete by induction, noting that we cannot have $\bar p(\kappa)+\bar q(\kappa)>\bar s(\kappa+1)$.
\end{proof}

\begin{proposition}\label{P: GM iter}
Let $P=\{\bar p_j\}_{j=1}^k$ be a collection of $n$-perversities such that $\sum_{j=1}^k \bar p_j(\kappa)\leq \bar r(\kappa)$ for all $2\leq \kappa\leq n$ and for some $n$-perversity $\bar r$. Let $X$ be an oriented $n$-dimensional pseudomanifold. Let $D_j\in I^{\bar p_j}C_{i_j}(X)$, $1\leq j\leq k$ be such that $\otimes_{j=1}^kS^{-n}D_j\in G_k^P$.
Then the iterated Goresky-MacPherson intersection product of the $D_j$ is a well-defined element of  
$I^{\bar r}C_{-n(k-1)+\sum i_j}(X)$, independent of arrangement of parentheses. In particular, 
there is a well-defined product $\prod_{j=1}^k IH_{i_j}^{\bar p_j}(X)\to IH^{\bar r}_{-n(k-1)+\sum i_j}(X)$ independent of arrangement of parentheses.
\end{proposition}
\begin{proof}
By \cite{GM1}, if $D_1\in I^{\bar p}C_a$ and $D_2\in I^{\bar q}C_b$ are in stratified general position and the boundary of $D_1$ is in stratified general position with respect to $D_2$ and vice versa, then there is a well-defined intersection product $D_1\pf D_2\in I^{\bar u}C_{a+b-n}(X)$ whenever $\bar p+\bar q\leq \bar u$. It follows from the preceding lemma that for any pair $D_{i_\ell}\times D_{i_\ell+1}\in IC_{i_\ell}^{\bar p_\ell}(X)\times I^{\bar pC_{\ell+1}}_{i_{\ell+1}}(X)$ such that $D_{i_\ell}$ and $D_{i_\ell+1}$ satisfy the necessary general position requirements, there is a well-defined pairing to $I^{\bar s}C_{i_\ell+i_{\ell+1}-n}(X)$, where $\bar s$ is the minimal $n$-perversity over $\bar p$ and $\bar q$. Since, by the lemma, $\bar s(\kappa)+\sum_{j\neq \ell, \ell+1} \bar p_j(\kappa)$ is still $\leq \bar r(\kappa)$, we can iterate the Goresky-MacPherson intersection product to obtain an $m$-fold intersection product so long as $D_{i_\ell}\pf D_{i_\ell+1}$ is in stratified general position (including the general position conditions on the boundaries) with whichever chain it will be intersected with next. But the condition $\otimes S^{-n} D_i\in G_k^P$ precisely guarantees that such general position will be maintained, even amongst combined sets of intersection (for any given surjective $R$ and any $i\neq j$, the intersection of the chains indexed by $R^{-1}(i)$ and the intersection of the chains indexed by $R^{-1}(j)$ will be in stratified general position by definition of $G_k^P$). Thus iteration is allowed.

The claim that this gives an iterated pairing on $IH$ follows immediately given that any two intersection cycles can be pushed into stratified general position within their homology classes - see \cite{GM1}. The claim concerning independence of ordering of parentheses is the claim that the the Goresky-MacPherson pairing is associative when the iterated pairing is well-defined. But this follows directly from the definition of the Goresky-MacPherson pairing and the associativity of the cup product: As noted in the proof above of Proposition \ref{P: GM}, $C\pf D$ is represented by $([C]\Upsilon\cup [D]\Upsilon)\cap \Gamma$ (we drop the signs in the duality isomorphisms since they cancel in the definition of $\pf$ - see the proof of Proposition \ref{P: GM}).
So the iterated product of $C$, $D$, and $E$ looks like 
\begin{align*}
([C]\pf[D])\pf[E]&=((([C]\Upsilon\cup [D]\Upsilon)\cap \Gamma)\Upsilon\cup [E]\Upsilon)\cap \Gamma\\
&=(([C]\Upsilon\cup [D]\Upsilon)\cup[E]\Upsilon)\cap \Gamma\\
&=([C]\Upsilon\cup ([D]\Upsilon\cup[E]\Upsilon)\cap \Gamma\\
&=([C]\Upsilon\cup (([D]\Upsilon\cup \Upsilon[E])\cap \Gamma)\Upsilon\cap \Gamma\\
&=[C]\pf([D]\pf [E]).
\end{align*} 
Note that in defining any of these products, we may use $J=|\bd C|\cup|\bd D|\cup |\bd E|\cup \Sigma$ (see the proof of Proposition \ref{P: GM}). Enlarging $J$ in this way will not interfere with the necessary excisions since, for example, having $S^{-n}C\otimes S^{-n}D\otimes S^{-n}E\subset G_3^P$ implies that $S^{-n}\bd E$ is in general position with respect to $S^{-n}C\cap S^{-n}D$. Thus $\dim(|\bd E|\cap |C|\cap |D|)<\dim(|C|)+\dim(|D|)-n$.

\end{proof}

\begin{corollary}
$\mu_2$ satisfies  $\mu_2(\mu_2(S^{-n}A\otimes S^{-n}B)\otimes S^{-n}C)=\mu_2(S^{-n}A\otimes \mu_2(S^{-n}B\otimes S^{-n}C))$ when these expressions are all well-defined.
\end{corollary}
\begin{proof}
This follows from the preceding proposition and Proposition \ref{P: GM}.
\end{proof}

Next, we compare how $\mu_k$ relates to the iteration of two products $\mu_{k_1}$ and $\mu_{k_2}$ with $k_1+k_1=k$.

\begin{lemma}\label{L: iteration}
Let $k=k_1+k_2$. Let $\bar p_a$, $1\leq a\leq k_1$, and $\bar p_{k_1+b}$, $1\leq b\leq k_2$ be collections of $n$-perversities such that $\sum_{a=1}^{k_1} \bar p_a\leq \bar q_1$ and $\sum_{b=1}^{k_2} \bar p_{k_1+b}\leq \bar q_2$ for perversities $\bar q_1,\bar q_2$. Suppose $\bar q_1+\bar q_2\leq \bar r$ for a perversity $\bar r$. Let  $P=(\bar p_1,\ldots, \bar p_k)$.
Suppose $C=S^{-n}D_1\otimes \cdots \otimes S^{-n} D_{k}$ is an element of $G_k^P$, and let $C_1=S^{-n}D_1\otimes \cdots \otimes S^{-n}D_{k_1}$ and $C_2=S^{-n}D_{k_1+1}\otimes \cdots \otimes S^{-n}D_{k}$. 
Then $$ \mu_{k}(C)= \mu_2( \mu_{k_1}(C_1)\otimes  \mu_{k_2}(C_2)).$$ In particular, $ \mu_k(C)=\mu_2( \mu_{k-1}(S^{-n}D_1\otimes\cdots\otimes S^{-n}D_{k-1})\otimes S^{-n}D_k)$. 
\end{lemma}
\begin{proof}
We first note that the righthand side of the desired equality is well defined since  the stratified general position requirements for any element of $G_k^P$ imply that for any $S^{-n}D_1\otimes \cdots \otimes S^{-n} D_{k}\in G^P_k$ and any disjoint subcollection $I, J\subset \{1,\ldots k\}$ then $\cap_{i\in I} |D_i|$ and $\cap_{j\in J} |D_j|$ are in stratified general position with respect to each other (and similarly for the necessary collections involving the $\bd D_i$). This can be seen by using the function $R:\bar k\onto  \overline{k-|I|-|J|+2}$ that takes $I$ to $1$, $J$ to $2$, and maps all other indices injectively. (Of course, we cannot in general split an  element of $G_k^P$ into an element of $G_{k_1}^{P_1}\otimes G_{k_2}^{P_2}$, but the element $C$ has an especially simple form.)

Let $\Delta_k:X\into X(k)$ be the diagonal embedding, let $\Gamma$ be the orientation class of $X$, let $\Gamma_k=\varepsilon_k(\Gamma\otimes \cdots\otimes \Gamma)=\Gamma\times \cdots \times \Gamma$, and  let $\Upsilon_k$ be the inverse Poincar\'e-Whitehead-Goresky-MacPherson duality isomorphism to the cap product with $\Gamma_k$. Let $\ell=\sum_i \dim(D_i)$. Then by definition, $ \mu_k(C)$ is represented by $S^{-n}(\Delta_k^*([\bar \varepsilon_k C]\Upsilon_k(-1)^{nk(nk-\ell)})\cap \Gamma (-1)^{n(nk-\ell)}$. Similarly,  letting $\ell_1=\sum_{i=1}^{k_1}\dim(D_i)$ and $\ell_2=\sum_{i=k_1+1}^k D_i$, then $\mu_2( \mu_{k_1}(C_1)\otimes  \mu_{k_2}(C_2))$ is represented by

\begin{align}\label{E: express}
S^{-n}&(\Delta_2^*((\bar \varepsilon_2
(S^{-n}(  
(
\Delta_{k_1}^*
(
[\bar \varepsilon_{k_1} C_1]\Upsilon_{k_1}(-1)^{nk_1(nk_1-\ell_1)}
)
)
\cap \Gamma (-1)^{n(nk_1-\ell_1)})\\
& 
\otimes S^{-n}((\Delta_{k_2}^*([\bar \varepsilon_{k_2} C_2]\Upsilon_{k_2}(-1)^{nk_2(nk_2-\ell_2)}))\cap \Gamma(-1)^{n(nk_2-\ell_2)})))\Upsilon_2(-1)^{2n(nk-\ell)}))\cap \Gamma (-1)^{n(nk-\ell)} \notag
\end{align}

Working mod $2$, the total power of the sign in this expression becomes $-1$ to the
\begin{align*}
nk_1(nk_1-\ell_1)&+n(nk_1-\ell_1)+nk_2(nk_2-\ell_2)+n(nk_2-\ell_2)+2n(nk-\ell)+n(nk-\ell)\\
&\equiv  nk_1-nk_1\ell_1+nk_1+n\ell_1+nk_2+nk_2\ell_2+nk_2+n\ell_2+nk+n\ell\\
&\equiv  nk_1\ell_1+nk_2\ell_2+nk
\end{align*} 
since $\ell_1+\ell_2=\ell$.

Since both of the formulas have the form $S^{-n}(\cdot)\cap \Gamma$, we can compare
$$\Delta_2^*((\bar \varepsilon_2
(S^{-n}(  
(
\Delta_{k_1}^*
(
[\bar \varepsilon_{k_1} C_1]\Upsilon_{k_1}
)
)
\cap \Gamma ) 
\otimes S^{-n}((\Delta_{k_2}^*([\bar \varepsilon_{k_2} C_2]\Upsilon_{k_2}))\cap \Gamma)))\Upsilon_2)$$ 
with
$\Delta_k^*([\bar \varepsilon C]\Upsilon_k).$

We compute 
\begin{align}\label{E: convert}
\Delta_2^*&((\bar \varepsilon_2(S^{-n}((\Delta_{k_1}^*([\bar \varepsilon_{k_1}C_1]\Upsilon_{k_1}
))\cap \Gamma )\otimes S^{-n}((\Delta_{k_2}^*([\bar \varepsilon_{k_2} C_2]\Upsilon_{k_2}))\cap 	
\Gamma)))\Upsilon_2)\\ 
&=(-1)^{n^2+n(n+\ell_1-nk_1)}\Delta_2^*((S^{-2n}(  
(\Delta_{k_1}^*([\bar \varepsilon_{k_1} C_1]\Upsilon_{k_1}))\cap \Gamma )
\times ((\Delta_{k_2}^*([\bar \varepsilon_{k_2} C_2]\Upsilon_{k_2}))\cap \Gamma))\Upsilon_2) \notag\\ \notag
&\qquad\qquad\qquad\qquad\qquad\qquad\qquad\qquad\qquad\qquad \text{def. of $\bve_2$}\\ \notag
&=(-1)^{n^2+n(n+\ell_1-nk_1)+n(nk_2-\ell_2)}\Delta_2^*((S^{-2n}(  
(\Delta_{k_1}^*([\bar \varepsilon_{k_1} C_1]\Upsilon_{k_1}))
\times (\Delta_{k_2}^*([\bar \varepsilon_{k_2} C_2]\Upsilon_{k_2})))\cap\Gamma_2\Upsilon_2))\\ \notag &\qquad\qquad\qquad\qquad\qquad\qquad\qquad\qquad\qquad\qquad\text{pulling $\cap \Gamma$ across}\\ \notag
&=(-1)^{n^2+n(n+\ell_1-nk_1)+n(nk_2-\ell_2)}\Delta_2^*(
((\Delta_{k_1}^*([\bar \varepsilon_{k_1} C_1]\Upsilon_{k_1}))
 \times (\Delta_{k_2}^*([\bar \varepsilon_{k_2} C_2]\Upsilon_{k_2}))))\\ \notag
 &\qquad\qquad\qquad\qquad\qquad\qquad\qquad\qquad\qquad\qquad\hfill\text{cancellation of $\Upsilon_2$ and $\cap \Gamma_2$}\\ \notag
 &=(-1)^{n^2+n(n+\ell_1-nk_1)+n(nk_2-\ell_2)}\Delta_2^*(\Delta_{k_1}^*\times \Delta_{k_2}^*)(([\bve_{k_1}C_1]\Upsilon_{k_1})\times ([\bve_{k_2}C_2]\Upsilon_{k_2}))\\ \notag
&=(-1)^{n^2+n(n+\ell_1-nk_1)+n(nk_2-\ell_2)} \Delta_{k}^*(([\bve_{k_1}C_1]\Upsilon_{k_1})\times ([\bve_{k_2}C_2]\Upsilon_{k_2})), 
 \end{align}
since $(\Delta_{k_1}\times \Delta_{k_2})\circ \Delta_2=\Delta_k$. The total sign here is $-1$ to the 

\begin{align*}
n^2+n(n+\ell_1-nk_1)+n(nk_2-\ell_2)\equiv n\ell+nk \mod 2.
\end{align*}

So it suffices to compare 
$([\bar \varepsilon_{k_1} C_1]\Upsilon_{k_1})\times ([\bar \varepsilon_{k_2} C_2]\Upsilon_{k_1})$ with $[\bar \varepsilon C]\Upsilon_{k}$.
Now suppose we include the signs that make $\Upsilon$ the inverse to the Poincar\'e duality morphism. In other words, we look at $([\bar \varepsilon_{k_1} C_1]\Upsilon_{k_1})(-1)^{nk_1(nk_1-\ell_1)}\times ([\bar \varepsilon_{k_2} C_2]\Upsilon_{k_1})(-1)^{nk_2(nk_2-\ell_2)}$. Then by Lemma \ref{L: eps dual} in Appendix A, this is equivalent to the cochain product of the individual inverse Poincar\'e duals of the individual chains. In other words, this is equal to $(S^{-n}D_1)\Upsilon_1(-1)^{n(n-|D_1|)}\times \cdots \times (S^{-n}D_k)\Upsilon_1(-1)^{n(n-|D_k|)}$, which, again by Lemma \ref{L: eps dual}, is equal to $\bve_k(C)\Upsilon_k(-1)^{nk(nk-\ell)}.$
Thus 
\begin{equation}\label{E: mix}([\bar \varepsilon_{k_1} C_1]\Upsilon_{k_1})\times ([\bar \varepsilon_{k_2} C_2]\Upsilon_{k_1})= (-1)^{nk_1(nk_1-\ell_1)+nk_1(nk_2-\ell_2)+nk(nk-\ell)}[\bve_kC]\Upsilon_k.
\end{equation}

This sign simplifies to $-1$ to the $nk_1\ell_1+nk_2\ell_2+nk\ell$.

Now, the total power of $-1$ in the expression \eqref{E: express} for $\mu_2( \mu_{k_1}(C_1)\otimes  \mu_{k_2}(C_2))$ is $nk_1\ell_1+nk_2\ell_2+nk$, the power of $-1$ from the computation \eqref{E: convert} is $n\ell+nk$, and the power of $-1$  from equation \eqref{E: mix} is $nk_1\ell_1+nk_2\ell_2+nk\ell$. Mod $2$, these add to $nkl+nl$, which is indeed equivalent mod $2$ to the power of $-1$ in the expression for $ \mu_k(C)$ with which we started. The lemma follows. 
\end{proof}

\begin{lemma}
Given chains $D_i$ as in the previous lemma, the iterated product $$ \mu_2( \mu_2(\cdots  \mu_2(S^{-n} D_1\otimes S^{-n}D_2)\otimes S^{-n}D_3)\otimes \cdots S^{-n}D_k)=    \mu_k(S^{-n}D_1\otimes \cdots \otimes S^{-n}D_k).$$
\end{lemma}
\begin{proof}
This follows directly from the  preceding lemma and induction.  
\end{proof}

\begin{proof}[Proof of Corollary \ref{C: iteration}]
Let $C_i=S^{-n}D_i$.  Since the Goresky-MacPherson pairing is associative, as noted in the proof of Proposition \ref{P: GM iter}, the arrangement of parentheses is immaterial, and we can use the grouping of the last lemma to consider  $((\cdots ((D_1\pf D_2)\pf D_3)\pf \cdots)\pf D_{k-1})\pf D_k$. By using Proposition \ref{P: GM}, repeatedly, this is equal to $
 S^n\mu_2( \mu_2(\cdots  \mu_2(C_1\otimes C_2)\otimes C_3)\otimes \cdots C_k)$, which, by the preceding lemma, is equal to 
$ S^n\mu_k(C_1\otimes \cdots \otimes C_k)$.

\end{proof}

\section{The Leinster partial algebra structure}\label{S: Leinster}

In this section, we collect the technical definitions concerning partial commutative DGAs and partial restricted commutative DGAs. Then we show that this is what we have, proving Theorem \ref{T: PRCA}, which was described in the introduction. 

The following definition without perversity restrictions originates from Leinster in \cite[Section 2.2]{Lei}, where the structures are referred to as \emph{homotopy algebras}. We follow McClure in \cite{McC}, where they are called \emph{partially defined DGAs} or \emph{Leinster partial DGAs}.

We continue to let $\bar k=\{1,\ldots,k\}$ for $k\geq 1$ and $\bar 0=\emptyset$. Let $\Phi$ be the full subcategory of $\mathbf{Set}$ consisting of the sets $\bar k$, $k\geq 0$. Note that disjoint union gives a functor $\amalg: \Phi\times \Phi\to \Phi$ determined by $\bar k\amalg \bar l=\overline{k+l}$. Given a functor $A$ with domain category $\Phi$, we denote $A(\bar k)$ by $A_k$.

\begin{definition}{(Leinster-McClure)}
A \emph{Leinster partial commutative DGA} is a functor $A$ from $\Phi$ to the category $\mathbf{Ch}$ of chain complexes
together with chain maps
$$\xi_{k,l}: A_{k+l}\to A_k\otimes A_l$$

for each $k,l$ and
$$
\xi_0: A_0\to \Z[0],$$
where $\Z[0]\in \mathbf{Ch}$ is the chain complex with a single $\Z$ term in degree $0$, such that the following conditions hold:

\begin{enumerate}
\item The collection $\xi_{k,l}$ is a natural transformation from 
$A\circ\amalg$ to $A\otimes A$, considered as functors from $\Phi\times \Phi$ 
to $\mathbf{Ch}$.

\item (Associativity) The diagram 
\begin{diagram}
A_{k+l+n} &\rTo^{\xi_{k+l,n}}& A_{k+l}\otimes A_n\\
\dTo^{\xi_{k,l+n}}&&\dTo_{\xi_{k,l}\otimes 1}\\
A_k\otimes A_{l+n} &\rTo^{1\otimes \xi_{l,n}}&A_k\otimes A_l\otimes A_n
\end{diagram}
commutes for all $k,l,n$.

\item (Commutativity) If $\tau: \overline{k+l}\to \overline{k+l}$ is the block permutation 
that transposes $\{1,\ldots,k\}$ and $\{k+1,\ldots,k+l\}$, then the following diagram commutes for all $k,l$:
\begin{diagram}
A_{k+l} &\rTo^{\xi_{k,l}} & A_k\otimes A_l\\
\dTo^{\tau_*}&&\dTo_{\cong}\\
A_{k+l} &\rTo^{\xi_{l,k}} &
A_l\otimes A_k.
\end{diagram}
(Note that the usual Koszul sign convention is in effect for the righthand isomorphism.)

\item (Unit) The diagram
\begin{diagram}
A_k &\rTo^{\xi_{0,k}}&A_{{0}}\otimes A_k\\
&\rdTo_{\cong}&\dTo_{\xi_0\otimes 1}\\
&& \Z[0] \otimes A_k
\end{diagram}
commutes for all $k$.

\item $\xi_0$ and each $\xi_{k,l}$ are quasi-isomorphisms.

\end{enumerate}

\end{definition}

The main theorem of McClure in \cite{McC} is that, given a compact oriented PL manifold $M$, there is a Leinster partial commutative DGA $G$ such that $G_k$ is a quasi-isomorphic subcomplex of the $k$-fold tensor product of PL chain complexes $S^{-n}C_*(M)\otimes \cdots \otimes S^{-n}C_*(M)$ and such that elements of $G_k$ represent chains in sufficient general position so that $G_k$ constitutes the domain of a $k$-fold intersection product.  
Notice the slightly subtle point that the intersection product itself is encoded in the fact that $G$ is a functor. Thus, for example, we have a map $G_k\to G_1=S^{-n}C_*(M)$, and this is precisely the intersection product coming from the umkehr map $\Delta_{k!}$.

For the intersection of intersection chains in a PL pseudomanifold, we must generalize to the notion of a partial restricted commutative DGA. In this setting,  the intersection pairing requires not just general position but compatibility among perversities. The appropriate generalized definition was suggested by Jim McClure and refined by Mark Hovey. 

Fix a non-negative integer $n$, we define a \emph{perverse chain complex} to be a functor from the poset category  $\mf P_n$ of $n$-perversities to the category $\mathbf{Ch}$ of chain complexes. The objects of $\mf P_n$ are $n$-perversities as defined in Definition \ref{D: n-perv}, and there is a unique morphism $\bar q\to \bar p$ if $\bar q(k)\leq \bar p(k)$ for all $k$, $2\leq k\leq n$.
We denote a perverse chain complex by $\{D^{\star}_*\}$. The $\star$ is meant to indicate the input variable for perversities, and we write evaluation as $\{D^{\star}_*\}^{\bar p}=D_*^{\bar p}$ or $\{D^{\star}_*\}^{\bar p}_i=D_i^{\bar p}$

This yields a category $\mathbf{PCh_n}$ of $n$-perverse chain complexes whose morphisms consist of natural transformations of such functors. Explicitly, given two perverse chain complexes $\{D^{\star}_*\}$ and $\{E^{\star}_*\}$, a morphism of perverse chain complexes consists of chain maps $D^{\bar p}_*\to E^{\bar p}_*$ for each perversity $\bar p$ together with commutative diagrams 

\begin{diagram}
D^{\bar q}_*&\rTo& E^{\bar q}_*\\
\dTo&&\dTo\\
D^{\bar p}_*&\rTo &E^{\bar p}_*,
\end{diagram}
whenever $\bar q\leq \bar p$.

We let $\{\Z[0]\}\in \mathbf{PCh_n}$ denote the perverse chain complex that at each perversity consists of a single $\Z$ term in degree $0$.

By \cite{Ho08}, a symmetric monoidal product $\boxtimes$ is obtained by setting  $(\{D^{\star}_*\}\boxtimes \{E^{\star}_*\})^{\bar r}=\displaystyle\dlim_{\bar p+\bar q\leq \bar r} D^{\bar p}_*\otimes E^{\bar q}_*$.  

\begin{definition}
A \emph{ Leinster partial restricted  commutative DGA} is a functor $A$ from $\Phi$ to the category $\mathbf{PCh_n}$ of $n$-perverse chain complexes (with images of objects denoted by $A(\bar k):=\{A^{\star}_{k,*}\}$), or simply $\{A^{\star}_{k}\}$ when we will not be working with individual degrees and no confusion will result,
together with morphisms
$$\zeta_{k,l}: \{A_{k+l}^{\star}\}\to \{A_k^{\star}\}\boxtimes \{A_l^{\star}\}$$

for each $k,l$ and
$$
\zeta_0: \{A_0^{\star}\}\to \{\Z[0]\},$$
such that the following conditions hold:

\begin{enumerate}
\item The collection $\zeta_{k,l}$ is a natural transformation from 
$\{A^{\star}\}\circ\amalg$ to $\{A^{\star}\}\boxtimes \{A^{\star}\}$, considered as functors from $\Phi\times \Phi$ 
to $\mathbf{PCh_n}$.

\item (Associativity) The diagram 
\begin{diagram}[LaTeXeqno]\label{D: assoc}
\{A_{k+l+n}^{\star}\} &\rTo^{\zeta_{k+l,n}}& \{A_{k+l}^{\star}\}\boxtimes \{A_n^{\star}\}\\
\dTo^{\zeta_{k,l+n}}&&\dTo_{\zeta_{k,l}\boxtimes 1}\\
\{A_k^{\star}\}\boxtimes \{A_{l+n}^{\star}\} &\rTo^{1\boxtimes \zeta_{l,n}}&\{A_k^{\star}\}\boxtimes \{A_l^{\star}\}\boxtimes \{A_n^{\star}\}
\end{diagram}
commutes for all $k,l,n$.

\item (Commutativity) If $\tau: \overline{k+l}\to \overline{k+l}$ is the block permutation 
that transposes $\{1,\ldots,k\}$ and $\{k+1,\ldots,k+l\}$, then the following diagram commutes for all $k,l$:
\begin{diagram}
\{A_{k+l}^{\star}\} &\rTo^{\zeta_{k,l}} & \{A_k^{\star}\}\boxtimes \{A_l^{\star}\}\\
\dTo^{\tau_*}&&\dTo_{\cong}\\
\{A_{k+l}^{\star}\} &\rTo^{\zeta_{l,k}} &
\{A_l^{\star}\}\boxtimes \{A_k^{\star}\}.
\end{diagram}

\item (Unit) The diagram
\begin{diagram}
\{A_k^{\star}\} &\rTo^{\zeta_{0,k}}&\{A_{{0}}^{\star}\}\boxtimes \{A_k^{\star}\}\\
&\rdTo_{\cong}&\dTo_{\zeta_0\boxtimes 1}\\
&& \{\Z[0]\} \boxtimes \{A_k^{\star}\}
\end{diagram}
commutes for all $k$.

\item $\zeta_0$ and each $\zeta_{k,l}$ are quasi-isomorphisms.

\end{enumerate}

\end{definition}

We can now restate Theorem \ref{T: PRCA} from the introduction and have it make some sense:

\begin{theorem}[Theorem \ref{T: PRCA}]
For any compact oriented PL stratified pseudomanifold $Y$, the partially-defined intersection pairing on the perverse chain complex $\{S^{-n}I^{\star}C_*(Y)\}$ extends to the structure of  a
 Leinster partial restricted commutative DGA.
\end{theorem}

So we must define an appropriate functor $A$ such that $\{A_1^{\star }\}\cong\{S^{-n}IC^{\star}_*(Y)\}$ and maps $\zeta_{k,l}$ and show that the conditions of the definition are satisfied. Furthermore, the $\{A_n^{\star }\}$ should be domains for appropriate intersection pairings, which which will be encoded within the functoriality.

To proceed,  let us say that a collection of $n$-perversities $P=\{\bar p_1,\ldots,\bar p_k\}$ satisfies $P\leq \bar r$ if $\sum_{i=1}^k\bar p_i(j)\leq \bar r(j)$ for all $j\leq n$.  Then we define a functor $\mf G:\Phi\to \mathbf{PCh_n}$ by letting $\mf G_0=\{\Z[0]\}$ and $$\{\mf G_k^{\star}\}(\bar r)=\dlim_{P\leq r} G_k^P,$$ with $G_k^P$ as defined above in Section \ref{S: GP}. This will be our functor ``$A$''. The fact that $\mf G$ is functorial on maps will be demonstrated below in the proof of the theorem.

For the definition of the $\zeta_{k,l}$, we will show in Proposition \ref{P: split}, deferred to below, that 
for two collections of perversities $P_1=\{\bar p_{1},\ldots, \bar p_{k}\}, P_2=\{\bar p_{k+1},\ldots, \bar p_{k+l}\}$, the inclusion $G_{k+l}^{P_1\amalg P_2}$ into the appropriate tensor product of terms $S^{-n}I^{\bar p_i}C_*(X)$ has its image in $G_k^{P_1}\otimes G_l^{P_2}$. Thus \begin{equation}\label{E: inclusion} G_{k+l}^{P_1\amalg P_2} \subset G_k^{P_1}\otimes G_l^{P_2}.\end{equation} Furthermore, as observed by Hovey \cite{Ho08}, the symmetric monoidal product on perverse chain complexes is associative in the strong sense that $$\{\{D^{\star}\} \boxtimes \{E^{\star}\}\boxtimes \{F^{\star}\}\}^{\bar r} \cong \dlim_{\bar p_1+\bar p_2+\bar p_3\leq \bar r} D^{\bar p_1}\otimes E^{\bar p_2}\otimes F^{\bar p_3},$$
independent of arrangement of parentheses, and similarly for products of more terms; the upshot of this is that any time we take a limit over tensor products of limits, it is equivalent to taking a single limit  over tensor products all at once. Thus, applying $\displaystyle \dlim_{\sum_{i}\bar p_{i}\leq \bar r}$ to \eqref{E: inclusion} and recalling that direct limits are exact functors, we obtain the inclusion of $\{\mf G_{k+l}^{\star}\}^{\bar r} $ in $\{\{\mf G_k^{\star}\} \boxtimes \{ \mf G_l^{\star}\}\}^{\bar r}$. Together, these give an inclusion $\zeta_{k,l}:\{\mf G_{k+l}^{\star}\}\into\{\{\mf G_k^{\star}\} \boxtimes \{ \mf G_l^{\star}\}\}$.

We now prove that $\mf G$, together with the maps $\zeta_{k,l}$, is a  Leinster partial restricted commutative DGA. 

\begin{proof}[Proof of Theorem \ref{T: PRCA}]
Assuming condition (1) of the definition  for the moment as well as continuing to assume Proposition \ref{P: split}, 
 in order to check the  other conditions of the definition, it is only necessary to check what happens for a specific set of perversities, since we can then apply the direct limit functor, which is exact. For example, given collections of perversities $P_1, P_2, P_3$ of length $k$, $l$, and $n$, condition (2) holds in the form
\begin{diagram}[LaTeXeqno]\label{D: local assoc}
G_{k+l+n}^{P_1\amalg P_2\amalg P_3} &\rTo& G_{k+l}^{P_1\amalg P_2}\otimes G_n^{P_3}\\
\dTo&&\dTo\\
G_k^{P_1}\otimes G_{l+n}^{P_2\amalg P_3} &\rTo& G_k^{P_1}\otimes G_l^{P_2}\otimes G_n^{P_3}.
\end{diagram}
This is clear from Proposition \ref{P: split} and the usual properties of tensor products. Now, to verify condition (2), we need only verify commutativity of  diagram \eqref{D: assoc} at each perversity $\bar r$, but the evaluation at $\bar r$ is simply the direct limit of diagram \eqref{D: local assoc}  over all collections $P_1,P_2,P_3$ with $P_1\amalg P_2\amalg P_3\leq \bar r$, using again Hovey's associativity property of the monoidal product.

Conditions (3) and (4) follow similarly from standard properties of tensor products, while condition (5) follows from Theorem \ref{T: pqi} and the exactness of the direct limit functor.

Now, for condition (1), we must first demonstrate the functoriality of $\mf G$, which means describing how $\mf G$ acts on maps $R:\bar k\to \bar l$. We abbreviate $\mf G(R)$ by $R_*$. Once again, we can start at the level of a specific $G_k^P$: Given $R$ and $G_k^P$, we must define $R_*:G_k^P\to G_l^{P'}$ for some collection $P'$ of perversities  such that if $P\leq \bar r$ then $P'\leq \bar r$. For each $G_k^P$ with $P\leq \bar r$, this gives us a legal composite map $G_k^P \to G_l^{P'}\to \dlim_{P'\leq \bar r}G_l^{P'}$. 
Once we do this in a way that is compatible with the inclusions $G_k^P\into G_k^Q$ when $P\leq Q\leq \bar r$ (meaning each perversity in $P$ is $\leq$ the corresponding perversity in $Q$), then $R_*:\{\mf G_k^{\star}\}\to \{\mf G_l^{\star}\}$ can be obtained by taking appropriate direct limits. 

So consider a set map $R:\bar k\to \bar l$. In \cite{McC}, McClure defines the morphism $R_*:G_k\to G_l$ on the groups associated to a manifold by proving that the composition $(R_!^*)\bar \epsilon_k$ has its image in $\bar \epsilon_l G_l$ so that defining $R_*$ by $\bar \epsilon_l^{-1}(R^*_!)\bar \epsilon_k$ makes sense.  Here $R^*_!$ is the transfer map associated to the generalized diagonal $R^*$; see \cite{McC} or Sections \ref{S: pairing} and \ref{S: GP}, above.
McClure's proof that we have well-defined maps $R_*:G_k\to G_l$ (from \cite[Section 10]{McC}) continues to hold in our setting so far as general position goes, so that  for pseudomanifolds and stratified general position, $\bar \epsilon_l^{-1}(R^*_!)\bar \epsilon_k$ is well-defined. However, we need next to take the perversities into account.

Any such $R:\bar k\to \bar l$ factors into a surjection, an injection, and permutations, so we can treat each of these cases separately. For permutations, $R=\sigma\in S_k$, we define $R_*$ on $G_k^P\subset S^{-n}I^{\bar p_1}C_*(X)\otimes \cdots \otimes S^{-n}I^{\bar p_k}C_*(X)$ by the (appropriately signed) permutation of terms as usual for tensor products. Since the defining stratified general position condition for $G_k^P$ is symmetric in all terms, the image will lie in $G_k^{\sigma P}$, where $\sigma P$ denotes the appropriately permuted collection of perversities. It is clear that if $P\leq \bar r$ then so is $\sigma P$ and also that this is functorial with respect to the inclusion maps in the poset of collections of perversities, and so $\sigma$ induces a well-defined homomorphism $\mf G_k\to \mf G_k$. 

Next, suppose that $R$ is an injection. Without loss of generality (since we have already considered permutations), we assume that $R(i)=i$ for all $i$, $1\leq i\leq k$. 
In this case, $R^*:X^l\to X^k$ is the projection onto the first $k$ factors. Given an element $\xi\in G_k^P\subset S^{-n}C_*(X)\otimes \cdots \otimes S^{-n}C_*(X)$, it is easy to check that, up to possible signs, $R_*(\xi)=\xi\otimes S^{-n} \Gamma\otimes \cdots S^{-n}\Gamma$, with $l-k$ copies of the shift of the fundamental orientation class $\Gamma$. But $\Gamma\in I^{\bar p}C_*(X)$ for any perversity, in particular for $\bar p=0$. So $R_*(\xi)\in S^{-n}I^{\bar p_1}C_*(X)\otimes \cdots \otimes S^{-n}I^{\bar p_k}C_*(X)\otimes S^{-n}I^{\bar 0}C_*(X)\otimes \cdots \otimes S^{-n}I^{\bar 0}C_*(X)$. Furthermore, since we have noted that stratified general position continues to hold under $R_*$, this must be an element of $G_l^{P\amalg^{l-k} \bar 0}$, where $P\amalg^{l-k} \bar 0=\{\bar p_1,\ldots,\bar p_k,\bar 0, \ldots,\bar 0\}$ adjoins $l-k$ copies of the $\bar 0$ perversity. 
Clearly $P\amalg^{l-k}\bar 0\leq \bar r$ if and only if  $P\leq \bar r$,
so indeed $R_*$ induces a map of $\mf G$. This is also clearly functorial with respect to the poset maps $P\leq Q$.

Finally, we have the case where $R$ is a surjection. All surjections can be written as compositions of permutations and surjections of the form $R(1)=R(2)=1$, $R(k)=k-1$ for $k>2$, so we will assume we have a surjection of this form. In this case, $R^*(x_1,x_2,\ldots,x_l)=(x_1,x_1,x_2,\ldots, x_l)$, and the intuition is that $R_*$ should correspond to the intersection product in the first two terms and the identity on the remaining terms. However, we must be careful to remember that the transfer $R_!$ does not necessarily give us a well-defined intersection map on primitives of the tensor product, only for chains in the tensor product satisfying the general position requirement, which may occur only due to certain cancellations amongst sums of primitives. So we must be careful to make sense of our intuition. Nonetheless, by Proposition \ref{P: split},  $G_k^P\subset G_2^{\bar p_1,\bar p_2}\otimes G_{k-2}^{\bar p_3,\ldots,\bar p(k)}$, so that $\xi\in G_k^P$ can be written as $\sum_{i,j} \eta_j\otimes \mu_i $, where $\eta_j\in G_2^{\bar p_1,\bar p_2}$. Writing $R=R_2\times \text{id}$, where $R_2:\bar 2\to \bar 1$ is the unique function,
it now makes sense that $R_*=R_{2*}\times \text{id}_*$ when applied to $\xi$, so that we obtain $R_*(\xi)=\sum R_{2*}(\eta_j)\otimes \mu_i$. Furthermore,  each $R_{2*}(\eta_j)$ will live in $S^{-n}I^{\bar s}C_*(X)$, where $\bar s$ is the minimal perversity over $\bar p_1$ and $\bar p_2$ (see Section \ref{S: GM}). So, $R_*(\xi)\in S^{-n}I^{\bar s}C_*(X)\otimes  S^{-n}I^{\bar p_3}C_*(X)\otimes \cdots\otimes S^{-n}I^{\bar p_k}C_*(X)$.
Applying Lemma \ref{L: perverse sum}, if $P\leq \bar r$ then $\bar s+\sum_{i\geq 3}\bar p_i\leq \bar r$. The image of $R_*$ is already known to satisfy the requisite stratified general position requirements (see above), and so $R_*$ induces a map from $G_k^P$ to $G_{k-1}^{\bar s,\bar p_3,\ldots,\bar p_k}$, which induces a map on $\mf G$.

We conclude that $\mf G$ is a functor.

The naturality of the $\zeta_{k,l}$ follows immediately: the only thing to check is compatible behavior between $\zeta_{k,l}$ and $\zeta_{k',l'}$ given two functions $R_1:\bar k\to \bar k'$ and $R_2:\bar l\to \bar l'$. But this is now easily checked since the $\zeta$ are inclusions and since the definitions of the maps $\mf G(R)=R_*$ are built precisely upon these inclusions and the ability to separate tensor products into different groupings, which is allowed by Proposition \ref{P: split}.
\end{proof}

Finally, we turn to  the deferred proposition showing that the maps $\zeta$ are induced by well-defined inclusions.

\begin{proposition}\label{P: split}
Let $P=\{\bar p_1,\ldots ,\bar p_{k+l}\}$, $P_1=\{\bar p_1,\ldots ,\bar p_{k}\}$, and $P_2=\{\bar p_{k+1},\ldots ,\bar p_{k+l}\}$. Then  $G_{k+l}^P\subset G_k^{P_1}\otimes G_l^{P_2}$.
\end{proposition}

We first need a lemma.

Let $\xi\in G_{k+l}^P\subset S^{-n}I^{\bar p_1}C_*(X) \otimes \cdots \otimes  S^{-n}I^{\bar p_{k+l}}C_*(X)$. We can write $\xi=\sum \xi_{i_1}\otimes \cdots \otimes \xi_{i_{k+l}}$, and we can fix a triangulation of $X$ with respect to which all possible $\xi_{i_j}$ are simplicial chains. (Note: we assume that $\xi_i\in S^{-n}I^{\bar p_i}C_*(X)$ rather than taking $\xi_i\in I^{\bar p_i}C_*(X)$ and then having to work with shifted chains $S^{-n}\xi$ for the rest of the argument; this leads to some abuse of notation in what follows, but this is preferable to dragging hordes of the symbol $S^{-n}$ around even more than necessary). Next, using that each $\xi_{i_j}$ is a sum $\xi_{i_j}=\sum b_{i_{j_k}}\sigma_k$, where the $\sigma_k$ are simplices of the triangulation, we rewrite $\xi$ as $\xi= \sum a_{i_{1}\ldots i_{k+l}}  \xi_{i_1}\otimes \cdots \otimes \xi_{i_k}\otimes \sigma_{i_{k+1}}\otimes \cdots \otimes \sigma_{i_{k+l}}$. In order to do this, we must of course consider $\xi$ as an element of $S^{-n}I^{\bar p_1}C_*(X)\otimes\cdots \otimes S^{-n} I^{\bar p_k}C_*(X)\otimes S^{-n}C_*(X)\otimes \cdots \otimes S^{-n}C_*(X)$. To help with the notation, we let $I$ be a multi-index of $k$ components, and we let $J$ be a multi-index of $l$ components. Then we can write $\xi=\sum_{I,J} a_{I,J}\xi_I\otimes \sigma_J$, where $a_{I,J}\in \Z$, $\xi_I\in S^{-n}I^{\bar p_1}C_*(X)\otimes \cdots\otimes S^{-n}I^{\bar p_k}C_*(X)$  and each $\sigma_J$ is a specific tensor product of simplices $\sigma_{i_{k+1}}\otimes \cdots \otimes \sigma_{i_{k+l}}$. 

Now, we fix a specific multi-index $J$ such that $\sum_I a_{I,J}\xi_I\otimes \sigma_J\neq 0$. Let $\eta_J = \sum_I a_{I,J}\xi_I$ (so $\xi=\sum_J \eta_J\otimes  \sigma_J$). 

\begin{lemma}
$\eta_J\in G_k^{P_1}$.
\end{lemma}
\begin{proof}
On the one hand, it is clear that each $\eta_J$ is a sum of tensor products of intersection chains, allowable with respect to the appropriate perversities. This is because in defining the $\eta_J$, we only split apart $\xi$ in the last $l$ slots, so that each $\eta_J$ is an appropriate sum of tensor products of chains $\xi_{i_j}$, $1\leq j\leq k$. 

On the other hand, $G_{k+l}^{P}\subset G_{k+l}$, and, by \cite[Lemma 11.1]{McC}, $G_{k+l}\subset G_k\otimes G_l$ (that argument is for manifolds, but works just as well here). So, as an element of $G_{k+l}$, $\xi$ can be rewritten as $\sum \mu_I\otimes \nu_J$, where $\mu_I\in G_k$ and $\nu_J\in G_l$. But now rewriting again by splitting all the $\nu_J$ up into tensor products of simplices, we recover $\xi=\sum_J \eta_J\otimes  \sigma_J$, but we now see that each $\eta_J$ can also be written as a sum of $\mu_I$s, each of which is in $G_k$. Hence each $\eta_J$ is in both $G_k$ and $S^{-n}I^{\bar p_1}C_*(X)\otimes \cdots \otimes S^{-n}I^{\bar p_k}C_*(X)$. Thus each is in $G_k^{P_1}$.  
\end{proof}

\begin{proof}[Proof of Proposition \ref{P: split}]
Consider the inclusions $i_1: G_k^{P_1}\into \otimes_{i=1}^k S^{-n}C_*(X)$ and $i_2: G_l^{P_2}\into \otimes_{i=1}^l S^{-n}C_*(X)$. Let $q_1$ be the projection $\otimes_{i=1}^k S^{-n}C_*(X) \to \cok(i_1)$ and similarly for $q_2$. Note that $\cok(i_1),\cok(i_2)$ are torsion free, since if any multiple of a chain $\xi$ is in stratified general position, then $\xi$ itself must also be in stratified general position and similarly for the allowability conditions defining the intersection chain complexes. Now, by basic homological algebra (see, e.g., \cite[Lemma 11.3]{McC}), $G_k^{P_1}\otimes G_l^{P_2}$ is precisely the kernel of $$q_1\otimes \text{id}+ \text{id}\otimes q_2:   \bigotimes_{i=1}^k S^{-n}C_*(X) \otimes  \bigotimes_{i=1}^l S^{-n}C_*(X)\to \left(\cok(i_1)\otimes \bigotimes_{i=1}^l S^{-n}C_*(X)\right)\oplus \left(\bigotimes_{i=1}^l S^{-n}C_*(X)\otimes \cok(i_2) \right) $$. 

So, if $\xi\in G_{k+l}^P$, we need only show that $\xi$ is in the kernel of this homomorphism, and it suffices to show that it is in the kernels of $q_1\otimes \text{id}$ and $\text{id}\otimes q_2$ separately. We will show the first; the argument for the second is the same.

So, we consider $\xi \in G_{k+l}^P\subset \bigotimes_{i=1}^k S^{-n}C_*(X) \otimes  \bigotimes_{i=1}^l S^{-n}C_*(X)$ and consider the image in $\cok(i_1)\otimes \bigotimes_{i=1}^l S^{-n}C_*(X)$ under $q_1\otimes \text{id}$. As above, we can rewrite $\xi$ here as $\sum \eta_J\otimes \sigma_J$. But now it follows from the preceding lemma that $\eta_J\in G_k^{P_1}$ and thus represents $0$ in $\cok(i_1)$. So $\xi\in \ker (q_1\otimes \text{id})$. 

By analogy, $\xi\in \ker ( \text{id}\otimes q_2)$, and we are done.
\end{proof}

\section{The intersection pairing in sheaf theoretic intersection homology}\label{S: sheaves}

In \cite{GM1}, Goresky and MacPherson defined the intersection homology intersection pairing  geometrically for  compact oriented PL pseudomanifolds. They used McCrory's theory of stratified general position \cite{Mc78} to show that any two PL intersection cycles are intersection homologous to cycles in stratified general position. The intersection of cycles is then  well-defined, and if $C\in I^{\bar p}C_*(X)$ and $D\in I^{\bar q}C_*(X)$ are in stratified general position, the intersection $C\pf D$  is in $I^{\bar r}C_*(X)$ for any $\bar r$ with $\bar r\geq \bar p+\bar q$. By \cite{GM2}, however, intersection homology duality was being realized on topological pseudomanifolds as a consequence of Verdier duality of sheaves in the derived category $D^b(X)$, and the intersection pairing was constructed via a sequences of extensions of morphisms from $X-\Sigma$ to all of $X$  (see also \cite{Bo}). The resulting morphism in $\Mor_{D^b(X)}(\mc I^{\bar p}\mc C_*\overset{L}{\otimes} \mc I^{\bar q}\mc C_*, \mc I^{\bar r}\mc C^*)$ indeed yields a pairing $I^{\bar p}H_i(X)\otimes I^{\bar q}H_j(X)\to I^{\bar r}H_{i+j-n}(X)$, but it is not completely obvious  that this pairing should agree with the earlier geometric one on PL pseudomanifolds. In this section, we demonstrate that these pairings do, indeed, coincide. While this is no doubt ``known to the experts,'' I know of no prior written proof. Furthermore, using the domain $G$ constructed above, we provide a ``roof'' in the category of sheaf complexes on $X$ that serves as a concrete representative of the derived category intersection pairing morphism.

\medskip

We first recall that, as noted in Remark \ref{R: gen sup},
our general position theorems of Section \ref{S: GP}
 hold just as well if we consider instead the complexes $C_*^{\infty}(X)$ and $I^{\bar p}C_*^{\infty}(X)$. In fact, the definitions of stratified general position carry over immediately, and all homotopies constructed in the proof of Theorems \ref{T: qi} and \ref{T: pqi}  are proper so that  they yield well-defined maps on these locally-finite chain complexes. The proofs that $G_k$ and $G_k^P$ are quasi-isomorphic to the appropriate tensor products is the same. We can also consider ``mixed type'' $G_k$s that are quasi-isomorphic to $S^{-n}I^{\bar p_1}C_*^{\infty}(X)\otimes \cdots \otimes S^{-n}I^{\bar p_j}C_*^{\infty}(X)\otimes S^{-n}I^{\bar p_{j+1}}C_*^{c}(X)\otimes \cdots \otimes S^{-n}I^{\bar p_k}C_*^{c}(X)$. 
 
For an open $U\subset X$, let $G^P_{k}(U)$ denote $G_k^P$ with respect to the pseudomanifold $U$.
Then $G_k^{P,\infty}$ is a contravariant functor from the category of open subsets of $X$ and inclusions to the category of chain complexes and chain maps. This is immediate, since if $C\in S^{-n}I^{\bar p_1}C^{\infty}_*(X)\otimes \cdots \otimes S^{-n}I^{\bar p_k}C^{\infty}_*(X)$ is such that $\bar \varepsilon_k(C)$ is in general position with respect to the appropriate diagonal maps, then certainly $\bar \varepsilon_k(C)|_{U(k)}=\bar \varepsilon_k(C|_U)$  maintains its general position, where $C|_U$ is the restriction of $C$ to $S^{-n}I^{\bar p_1}C^{\infty}_*(U)\otimes \cdots \otimes S^{-n}I^{\bar p_k}C^{\infty}_*(U)$. It is also clear that such restriction is functorial. 
Let $\mc G^{P}_{k,*}$ be the sheafification of the presheaf $G_{k,*}^{P,\infty}: U\to G_{k,*}^{P,\infty}(U)$. Note that the $k$ here denotes the number of terms in the tensor product, while $*$ is the  dimension index.

Let $\mc I^{P}\mc C_{*}$ be the sheafification of the presheaf $I^PC_{*}:U\to S^{-n}I^{\bar p_1}C^{\infty}_{*}(U)\otimes \cdots \otimes S^{-n}I^{\bar p_k}C^{\infty}_{*}  (U)$. This is the  tensor product of the sheaves $\mc I^{\bar p_i}\mc C_{*}$ with $\mc I^{\bar p_i}\mc C_{*}(U)=S^{-n}I^{\bar p_i}C^{\infty}_{*}(U)$, which are the (degree shifted) intersection chain sheaves of \cite{GM2} (see also \cite[Chapter II]{Bo}). 

\emph{N.B. We build the shifts into the definitions of all sheaves and of the presheaf $IC^P_*$ (which, after all, is not unusual in intersection cohomology with certain indexing schemes - see \cite{GM2}). However, for chain complexes of a single perversity, $I^{\bar p}C_*(X)$  continues to denote the unshifted complex, and we write in any shifts as necessary.}

\begin{lemma}\label{L: qi}
The inclusion of presheaves $G_{k,*}^{P,\infty}\into I^PC^{\infty}_{*}$ induces a quasi-isomorphism of sheaves $\mc G^P_{k,*}\to \mc I^{P}\mc C_*$.
\end{lemma}
\begin{proof}
By the results of Section  \ref{S: GP}, each inclusion $G_{k,*}^{P,\infty}(U)\into I^PC^{\infty}_{*}(U)$ is a quasi-isomorphism. Taking direct limits over neighborhoods of each point $x\in X$ therefore yields isomorphisms of stalk cohomologies.
\end{proof}

\begin{corollary}
For any system of supports $\Phi$, the sheaf map of the lemma induces a hyperhomology isomorphism for each $U$,
$\H^{\Phi}_{*}(U;\mc G^{P}_{k,*})\to \H^{\Phi}_*(U;\mc I^{P}\mc C_{*}).$
\end{corollary}

\begin{proposition}\label{P: product}
If $\sum_i \bar p_i\leq \bar r$, then the intersection product $\mu_k$ induces a sheaf map  $m: \mc G^P_{k,*}\to \mc I^{\bar r}\mc C_*$. 
\end{proposition}
\begin{proof}
We need only note that the intersection product $\mu_k:G_k^{ P,\infty}(U)\to S^{-n}I^{\bar r}C^{\infty}_*(U)$  behaves functorially under restriction. Thus, it induces a map of presheaves, which induces the map of sheaves. 
\end{proof}

\begin{lemma}
On $X-\Sigma$, the sheaf map $m$ of the preceding proposition is quasi-isomorphic to  the standard product map $\phi:\Z|_{X-\Sigma}\otimes\cdots\otimes \Z|_{X-\Sigma}\to \Z|_{X-\Sigma}$. 
\end{lemma}
\begin{proof}
We first observe that
$\mc G_{k,*}^P|_{X-\Sigma}\sim_{q.i.} \Z|_{X-\Sigma}\otimes\cdots\otimes \Z|_{X-\Sigma}$: 
\begin{align*}
\mc G_{k,*}^P|_{X-\Sigma}&\sim_{q.i.}\mc I^{P}\mc C_{*}|_{X-\Sigma}\\
&=  \mc I^{\bar p_1}\mc C_{*}|_{X-\Sigma}\otimes \cdots \otimes\mc I^{\bar p_k}\mc C_{*}|_{X-\Sigma}\\
&\sim_{q.i.}\Z|_{X-\Sigma}\otimes\cdots\otimes \Z|_{X-\Sigma}.
\end{align*}

Now, for each point $x\in X-\Sigma$, let $U\cong \R^n$ be a euclidean neighborhood. Then a generator of $\Z_U\otimes\cdots\otimes \Z_U$ corresponds to $S^{-n}\mc O\otimes \cdots \otimes S^{-n}\mc O$, where $\mc O$ is  the $n$-dimensional orientation cycle for $\R^n$ in $C^{\infty}_n(U)$. But $\bar \varepsilon_k(S^{-n}\mc O\otimes \cdots S^{-n} \mc O)$ is  automatically in stratified general position with respect to the diagonal $\Delta$ by dimension considerations; thus $S^{-n}\mc O\otimes \cdots\otimes S^{-n}\mc O$  is in $G_{k,0}^{P,\infty}$. Furthermore the image of $S^{-n}\mc O\otimes \cdots\otimes S^{-n}\mc O$ under $\mu_k$ is again $S^{-n}\mc O\in  S^{-n}C^{\infty}_n(U)$, which corresponds to a generator of $\Z_U\sim_{q.i.}S^{-n}\mc I^{\bar r}\mc C_*|U$. This can be seen by considering $\mc O\pf\cdots\pf \mc O$, which is classically equal to $\mc O$, and by using the results of the preceding Section. 

Since all of the above is compatible with restrictions and induces isomorphisms on the restrictions from $\R^n$ to $B^n$ (for any open ball $B^n$ in $\R^n$), and since it is this map of presheaves that induces the map of sheaves we are considering, the lemma follows.
\end{proof}

\begin{lemma}\label{L: flat}
$\mc I\mc C_*$ is a flat sheaf.
\end{lemma}
\begin{proof}
For each $i$ and each open $U\subset X$, $IC_i^{\infty}(U)$ is torsion free. Thus tensor product of abelian groups with  $IC_i^{\infty}(U)$ is an exact functor, and thus $\mc I\mc C_*$ is flat as a presheaf. Taking direct limits shows that tensor product with $\mc I\mc C_*$ is exact as a functor of sheaves. So $\mc I\mc C_*$ is flat.
\end{proof}

We now limit ourselves to considering $G^P_2$ with various supports. 

Let $\mc P^{\bar p}_*$ be the perversity $\bar p$ Deligne sheaf (see \cite{GM2,Bo}), reindexed to be compatible with our current homological notation. 
According to \cite[Proposition V.9.14]{Bo}, there is in $D^b(X)$ a unique morphism $\Phi:\mc P^{\bar p}_*\overset{L}{\otimes} \mc P^{\bar q}_*\to \mc P^{\bar r}_*$ that extends the multiplication morphism $\phi:\Z_{X-\Sigma}\otimes \Z_{X-\Sigma}\to \Z_{X-\Sigma}$. Since $\mc I^{\bar p}\mc C_*$ is quasi-isomorphic to $\mc P^{\bar p}_*$ by \cite{GM2} and flat by Lemma \ref{L: flat}, the tensor complex $\mc I^{\bar p}\mc C_*\otimes \mc I^{\bar q}\mc C_*$ represents $\mc P^{\bar p}_*\overset{L}{\otimes} \mc P^{\bar q}_*$ in $D^b(X)$, and we can represent morphisms $\mc P^{\bar p}_*\overset{L}{\otimes} \mc P^{\bar q}_*\to \mc P^{\bar r}_*$ in $D^b(X)$ by roofs in the category of sheaf complexes
\begin{equation}\label{E: roof}\mc I^{\bar p}\mc C_*\otimes \mc I^{\bar q}\mc C_* \overset{s}{\leftarrow}\mc S_* \overset{f}{\to} \mc I^{\bar r}\mc C_*,
\end{equation}
where $f$ is a sheaf morphism and $s$ is a sheaf quasi-isomorphism. For the duality product morphism, we set  $\mc S_*$ equal to $\mc G^{P}_{2,*}$, and let $f$ be the sheaf map $m$ of Proposition \ref{P: product} and $s$ the quasi-isomorphism of Lemma \ref{L: qi}. We will show that the restriction of this roof to $X-\Sigma$ is equivalent to $\phi:\Z_{X-\Sigma}\otimes \Z_{X-\Sigma}\to \Z_{X-\Sigma}$ in $D^b(X-\Sigma)$.  

\begin{proposition}
Under the isomorphism $$\Mor_{D^b(X-\Sigma)}((\mc I^{\bar p}\mc C_*\otimes \mc I^{\bar q}\mc C_*)|_{X-\Sigma} ,  \mc I^{\bar r}\mc C_*|_{X-\Sigma})\cong \Mor_{D^b(X-\Sigma)}(\Z_{X-\Sigma}\otimes \Z_{X-\Sigma},\Z_{X-\Sigma}),$$ the restriction of the roof 
\begin{equation}\label{E: roof3}\mc I^{\bar p}\mc C_*\otimes \mc I^{\bar q}\mc C_* \overset{\sim_{q.i}}{\leftarrow}\mc G^P_{2,*} \overset{m}{\to} \mc I^{\bar r}\mc C^*,
\end{equation}
 to $X-\Sigma$ corresponds to the standard multiplication morphism $\phi:\Z_{X-\Sigma}\otimes \Z_{X-\Sigma}\to \Z_{X-\Sigma}$.
 \end{proposition}
\begin{proof}
$\phi$ is represented in $\Mor_{D(X-\Sigma)}(\Z_{X-\Sigma}\otimes \Z_{X-\Sigma},\Z_{X-\Sigma})$ by the roof
$$\Z_{X-\Sigma}\otimes \Z_{X-\Sigma}\overset{=}{\leftarrow} \Z_{X-\Sigma}\otimes \Z_{X-\Sigma}\overset{\phi}{\to} \Z_{X-\Sigma}.$$
To identify this with an element of 
\begin{equation}\label{E: roof2}
\Mor_{D^b(X-\Sigma)}((\mc I^{\bar p}\mc C_*\otimes \mc I^{\bar q}\mc C_*)|_{X-\Sigma} ,  \mc I^{\bar r}\mc C_*|_{X-\Sigma}),
\end{equation}
 which is isomorphic to $\Mor_{D^b(X-\Sigma)}(\Z_{X-\Sigma}\otimes \Z_{X-\Sigma},\Z_{X-\Sigma})$  due to the quasi-isomorphisms of the sheaves involved, we must pre- and post-compose in $D^b(X-\Sigma)$ with the appropriate $D^b(X-\Sigma)$ isomorphisms. These can be represented as roofs
$$ (\mc I^{\bar p}\mc C_*\otimes \mc I^{\bar q}\mc C_*)|_{X-\Sigma} \overset{F'}{\leftarrow}\Z_{X-\Sigma}\otimes \Z_{X-\Sigma} \overset{=}{\to}\Z_{X-\Sigma}\otimes \Z_{X-\Sigma}$$ and 
$$ \Z_{X-\Sigma}\overset{=}{\leftarrow}\Z_{X-\Sigma} \overset{F}{\to}(\mc I^{\bar r}\mc C_*)|_{X-\Sigma}.$$ 
The map $F$ is induced by taking $z\in \Gamma(X-\Sigma;\Z_{X-\Sigma})\cong \Z$ to $z$ times the orientation class $\mc O$, and $F'$ takes $y\otimes z$ to $yz$ times the image of $\mc O\times \mc O$ in the sheaf $(\mc I^{\bar p}\mc C_*\otimes \mc I^{\bar q}\mc C_*)|_{X-\Sigma}$. 

Some routine roof equivalence arguments  yield that $\phi$, together with the pre- and post-compositions of isomorphisms, is equivalent to the roof
$$ (\mc I^{\bar p}\mc C_*\otimes  \mc I^{\bar q}\mc C_*)|_{X-\Sigma} \overset{F'}{\leftarrow}\Z_{X-\Sigma}\otimes \Z_{X-\Sigma} \overset{H}{\to} (\mc I^{\bar r}\mc C_*)|_{X-\Sigma},$$ 
where $H$ is the composition of $\phi$ and $F$. 

To see that this last roof is equivalent to the restriction of \eqref{E: roof3}  to $X-\Sigma$, we need only note that $F'$ factors through $\mc G_{k,*}^P|_{X-\Sigma}$, since $\mc O$ is in general position with respect to itself, and that the  composition $F'': \Z_{X-\Sigma}\otimes \Z_{X-\Sigma}\to \mc G_{k,*}^P|_{X-\Sigma}\to (\mc I^{\bar r}\mc C_*)|_{X-\Sigma}$ is precisely the same multiple of the orientation class that we get from $F\circ \phi$. 

Thus we have demonstrated the proposition.
 \end{proof}

 \begin{corollary}\label{C: ext}
 The morphism in $\Mor_{D^b(X)}(\mc I^{\bar p}\mc C_*\otimes \mc I^{\bar q}\mc C_* ,  \mc I^{\bar r}\mc C_*)$ represented by the
  roof \eqref{E: roof3} must be the unique extension from $\Mor_{D^b(X-\Sigma)}((\mc I^{\bar p}\mc C_*\otimes \mc I^{\bar q}\mc C_*)|_{X-\Sigma} ,  \mc I^{\bar r}\mc C^*|_{X-\Sigma})$ of the image of the multiplication $\phi$ under the isomorphism
 $ \Mor_{D^b(X-\Sigma)}(\Z_{X-\Sigma}\otimes \Z_{X-\Sigma},\Z_{X-\Sigma})      \to \Mor_{D^b(X-\Sigma)}((\mc I^{\bar p}\mc C_*\otimes \mc I^{\bar q}\mc C_*)|_{X-\Sigma} ,  \mc I^{\bar r}\mc C_*|_{X-\Sigma})$.
\end{corollary}
\begin{proof}
From the proposition, the roof \eqref{E: roof3} restricts to a morphism corresponding to $\phi$ on $X-\Sigma$. The uniqueness follows as in \cite[Proposition V.9.14]{Bo}.
\end{proof}

Finally, we can show that the geometric intersection pairing is isomorphic to the sheaf-theoretic pairing.

\begin{theorem}
If $\bar p+\bar q\leq \bar r$, then the pairings
\begin{align*}
I^{\bar p}H_i^{\infty}(X)\otimes I^{\bar q}H_{j}^{\infty}(X)&\to I^{\bar r}H_{i+j-n}^{\infty}(X)\\
I^{\bar p}H_i^c(X)\otimes I^{\bar q}H_{j}^{c}(X)&\to I^{\bar r}H_{i+j-n}^c(X)\\
I^{\bar p}H_i^c(X)\otimes I^{\bar q}H_{j}^{\infty}(X)&\to I^{\bar r}H_{i+j-n}^c(X)
\end{align*} 
determined by sheaf theory are isomorphic to the respective pairings determined by geometric intersection.
\end{theorem}
\begin{proof}
From \cite[Section V.9]{Bo}, the sheaf theoretic pairing is induced by the unique extension of  the morphism $\phi:\Z_{X-\Sigma}\otimes \Z_{X-\Sigma}\to \Z_{X-\Sigma}$
in $\Mor_{D^b(X-\Sigma)}(\Z_{X-\Sigma}\otimes \Z_{X-\Sigma},\Z_{X-\Sigma})$ to $\Mor_{D^b(X)}(\mc I^{\bar p}\mc C_*\overset{L}{\otimes}\mc I^{\bar q}\mc C_*, \mc I^{\bar r}\mc C_*)$. Given this unique extension, which we shall denote $\pi$, the intersection homology pairings can be described as follows. Since the intersection chain sheaves are soft (see \cite[Chapter II]{Bo}), a generating element $s\otimes t\in IH_i^{\Phi}(X)\otimes IH_{j}^{\Psi}(X)$ (where $\Phi$ and $\Psi$ represent $c$ or $\infty$) is represented by sections $s\in \Gamma_\Phi(X;\mc I^{\bar p}\mc C_{i-n})$ and $t\in \Gamma_\Psi(X;\mc I^{\bar q}\mc C_{j-n})$ such that $\bd s=\bd t=0$ as sections. 
Since $(\mc I^{\bar p}\mc C_{i-n}\otimes \mc I^{\bar q}\mc C_{j-n})_x\cong (\mc I^{\bar p}\mc C_{i-n})_x\otimes (\mc I^{\bar p}\mc C_{j-n})_x$, $s\otimes t$ determines  a  section of  $\Gamma(X;\mc I^{\bar p}\mc C_{i-n}\otimes \mc I^{\bar q}\mc C_{j-n})$, which is isomorphic to $\Gamma(X;\mc I^{\bar p}\mc C_{i-n}\overset{L}{\otimes} \mc I^{\bar q}\mc C_{j-n})$ by Lemma \ref{L: flat}. If  either $s$ or $t$ has compact support, so does $s\otimes t$. This section then maps to a cycle in any injective resolution of $\mc I^{\bar p}\mc C_{*}\overset{L}{\otimes}\mc I^{\bar q}\mc C_{*}$  and thus represents an element $z$ in the hyperhomology $\H_{i+j-2n}(X;\mc I^{\bar p}\mc C_{*}\overset{L}{\otimes} \mc I^{\bar q}\mc C_{*})$. If $s\otimes t$ has compact support, $z$ also represents an element of $\H_{i+j-2n}^c(X;\mc I^{\bar p}\mc C_{*}\overset{L}{\otimes} \mc I^{\bar q}\mc C_{*})$.

Now, due to Corollary \ref{C: ext}, the morphism $\pi$ is represented by the roof
$$\mc I^{\bar p}\mc C_*\otimes \mc I^{\bar q}\mc C_* \overset{q.i.}{\leftarrow}\mc G_{2,*}^P \overset{m}{\to} \mc I^{\bar r}\mc C_*,$$
which induces hyperhomology morphisms
$$\H^\Phi_*(X;\mc I^{\bar p}\mc C_*\otimes \mc I^{\bar q}\mc C_*) \overset{\cong}{\leftarrow}\H^\Phi_*(X;\mc G_{2,*}^P) \to\H^\Phi_*(X; \mc I^{\bar r}\mc C_*).$$
Making the desired choices of supports and applying to $z$ the composition of the inverse of the lefthand isomorphism and the righthand morphism gives the pairings as defined via sheaf theory.

 Now, consider the following diagram. For the moment, we take $\Phi=\Psi$, which can be either $c$ or $\infty$.

\begin{diagram}[LaTeXeqno]\label{E: diagram}
H_*(S^{-n}I^{\bar p}C^\Phi_*(X)\otimes S^{-n}I^{\bar q}C^{\Psi}_*(X))& \lTo^{\cong}&H_*(G^{P,\Phi}_{2,*})&\rTo&H_*(S^{-n}I^{\bar r}C^\Phi_*(X))\\
\dTo&&\dTo&&\dTo^{\cong}\\
\H^\Phi_*(X;\mc I^{\bar p}\mc C_*\otimes \mc I^{\bar q}\mc C_*)& \lTo^{\cong}&\H^\Phi_*(X;\mc G_{2,*}^P) &\rTo&\H^\Phi_*(X; \mc I^{\bar r}\mc C_*).
\end{diagram}
The groups on the top row are simply the homology groups of the sections of presheaves with supports in $\Phi$. The vertical homology maps are induced by taking presheaf sections to sheaf sections to sections of injective resolutions. Since sheafification and injective resolution are natural functors, the diagram commutes. Applied to the tensor product of two chains in stratified general position,  the composition of the lefthand vertical map with the maps of the bottom row is exactly the sheaf theoretic pairing as described above. 
Meanwhile, the composition of maps along the top row is  the geometric pairing $\mu_2$ defined above using the domain $G_2^{P,\infty}$. 
The theorem now follows in this case from the commutativity of the diagram and the results of the previous sections, in which we demonstrated that, for a pair of chains in stratified general position,  $\mu_2$ agrees with the Goresky-MacPherson product. 

When $\Phi=c$ and $\Psi=\infty$, we must be a bit more careful. Here we replace $H_*(G^{P,\Phi}_{2,*})$ with the homology of the subcomplex $\hat G^P_{2,*}(X)\subset G^{P,\infty}_{2,*}(X)$ defined as follows. Recall that $G^{P,\infty}_{2,*}(X)$ is a subcomplex of $S^{-n}I^{\bar p}C^{\infty}_*(X)\otimes S^{-n}I^{\bar q}C_*^{\infty}(X)$. Thus any element $e\in G^P_{2,*}(X)$ can be written as a finite sum $e=\sum S^{-n}\xi_i\otimes S^{-n} \eta_i$, where $\xi_i\in I^{\bar p}C^{\infty}_*(X)$ and $\eta_i\in I^{\bar q}C_*^{\infty}(X)$. We let $\hat G^P_{2,*}(X)$ consist of such sums for which each $\xi_i$ has compact support. This is clearly a subcomplex, and the  general position proof of Section \ref{S: GP} shows that  $\hat G^P_{2,*}(X)$ is quasi-isomorphic to $S^{-n}I^{\bar p}C^{c}_*(X)\otimes S^{-n}I^{\bar q}C_*^{\infty}(X)$.  We also observe that the image of each such element of $\hat G^P_{2,*}(X)$ in the sheaf $\mc G^P_{2,*}(X)$ has compact support. Indeed, if $x\notin \cup |\xi_i|$, which is compact, then the restriction of $e$ to a neighborhood $U$ of $x$ must have the form $\sum S^{-n}0\otimes S^{-n} \eta_i|_U=0$. 

Now we can take diagram \eqref{E: diagram} with $\Phi=c$, $\Psi=\infty$ and with $G^{P,\infty}_{2,*}(X)$ replaced by $\hat G_{2,*}^{P}(X)$ in the middle of the top row. The diagram continues to commute, and the correspondence between the geometric and sheaf-theoretic pairings follows as for the preceding cases.
\end{proof}

As a result of the theorem, several common practices become easily justified. For example, we can demonstrate that the sheaf theoretic product has a symmetric middle-dimensional pairing for  oriented Witt spaces of dimension $0\mod 4$ and an anti-symmetric middle-dimensional pairing for  oriented Witt spaces of dimension $2 \mod 4$. To see this, we note that, if $C\in I^{\bar p}C_i(X)$ and $D\in I^{\bar q}C_j(X)$ with $\bar p+\bar q\leq \bar r$ for some $\bar r$ and $C$ and $D$ in stratified general position, then 

\begin{align*}
S^n\mu_2(S^{-n} C,S^{-n}D)&=C\pf D\\
&=(-1)^{(n-i)(n-j)}D\pf C\\
&=(-1)^{(n-i)(n-j)}S^n\mu_2(S^{-n}D,S^{-n}C).
\end{align*}
The second equality here uses the well-known graded symmetry of geometric intersection products.
So, in particular, if $X$ is a Witt space and $\bar p=\bar q=\bar m$, the lower middle perversity, and if $n=4w$ and $i=j=2w$, then the product is symmetric. Similarly, if $n=2w\equiv 2\mod 4$ and $i=j=w$, then the pairing is anti-symmetric.

Of course this is well-known for geometric intersection products, but it is not completely obvious from Verdier duality (see, e.g., \cite[Appendix]{Ba02}).

\section{Appendix A - Sign issues}

In this appendix we collect some technical lemmas, especially those that correct the sign issues in the original version of \cite{McC}. We refer the reader to the main text above for some of the definitions and also to the revised version of \cite{McC}. The sign corrections  necessary to perform these computations are due to McClure.

Recall that for complexes $A^i_*$, we define $\Theta: S^{m_1}A^1_*\otimes \cdots \otimes S^{m_k}A^k_* \to S^{\sum m_i} (A_1^* \otimes \cdots \otimes A^k_*)$ by $$\Theta(S^{m_1}x_1\otimes \cdots \otimes S^{m_k}x_k)=(-1)^{\sum_{i=2}^k (m_i\sum_{j<i}|x_j|)} S^{\sum m_i}(x_1\otimes \cdots \times x_k).$$

\begin{lemma}\label{L: Theta}
$\Theta: S^{m_1}A^1_*\otimes \cdots \otimes S^{m_k}A^k_* \to S^{\sum m_i} (A_1^* \otimes \cdots \otimes A^k_*)$ is a chain isomorphism.
\end{lemma}
\begin{proof}
We compute 
\begin{align*}
\bd \Theta(S^{m_1}x_1\otimes \cdots \otimes S^{m_k}x_k)&=\bd(-1)^{\sum_{i=2}^k (m_i\sum_{j<i}|x_j|)} S^{\sum m_i}(x_1\otimes \cdots \times x_k)\\
&=\sum_l (-1)^{\sum_{i=2}^k (m_i\sum_{j<i}|x_j|)+\sum m_i} S^{\sum m_i} x_1\otimes \cdots \otimes (-1)^{\sum_{a<l}|x_a|}\bd x_l \otimes \cdots\otimes x_k\\
&=\sum_l (-1)^{\sum_{i=2}^k (m_i\sum_{j<i}|x_j|)+\sum m_i+\sum_{a<l}|x_a|} S^{\sum m_i} x_1\otimes \cdots \otimes \bd x_l \otimes \cdots\otimes x_k,
\end{align*}
while 

\begin{align*}
\Theta \bd &(S^{m_1}x_1\otimes \cdots \otimes S^{m_k}x_k)= 
\Theta(\sum_l    S^{m_1}x_1\otimes \cdots \otimes (-1)^{\sum_{a<l}|x_a|+\sum_{b\leq l}m_b}  S^{m_l}\bd x_l\otimes \cdots\otimes S^{m_k}x_k)\\
&= \sum_l (-1)^{\sum_{a<l}|x_a|+\sum_{b\leq l}m_b}
(-1)^{ \sum_{r\leq l}(m_r(\sum_{j<r}|x_j|))+\sum_{s>l}(m_s(-1+\sum_{j<s} |x_j|))}    S^{\sum m_i} x_1\otimes \cdots \otimes \bd x_l\otimes \cdots \otimes x_k
\end{align*}

It is not difficult to compare the two signs and see that they agree. Therefore $\Theta$ is a chain map. It is clearly an isomorphism.

\end{proof}

Recall from Section \ref{S: GP}   that $\bar \varepsilon: S^{-m_1}C_*(M_1)\otimes \cdots \otimes S^{-m_k}C_*(M_k) \to S^{-\sum m_i}C_*(M_1\times \cdots \times M_k)$ is defined to be $(-1)^{e_2(m_1,\ldots ,m_k)}$ times the composition of $\Theta$ with the $S^{-\sum m_i}$ shift of McClure's chain product $\varepsilon$.

\begin{lemma}\label{L: eps dual}
$\bar \varepsilon_k$ is dual to the iterated cochain cross product under the (signed) Poincar\'e duality morphism. In other words, letting $P_{X_i}$ be the Poincar\'e duality map on the oriented $m_i$-pseudomanifold $X_i$, given by the appropriately signed cap product with the fundamental class $\Gamma_{X_i}$ and shifted to be a degree $0$ chain map, there is a commutative diagram
\begin{diagram}
C^{-*}(X_1)\otimes\cdots \otimes  C^{-*}(X_k)&\rTo^{\times\cdots \times  }&C^{-*}(X_1 \times\cdots \times X_k)\\
\dTo^{P_{X_1}\otimes \cdots \otimes P_{X_k}}&&\dTo_{P_{X_1\times \cdots \times X_k}}\\
S^{-m_1}C_*(X_1)\otimes \cdots \otimes S^{-m_k}C_*(X_k)&\rTo^{\bar \varepsilon_k}&S^{-\sum m_i}C_*(X_1\times \cdots \times X_k).
\end{diagram}
\end{lemma}

\begin{proof}
Let $x_i\in C^{-*}(X_i)$ be homogeneous elements of degree $|x_i|$. Then $(P_{X_1}\otimes \cdots \otimes  P_{X_k})(x_1\otimes \cdots \otimes  x_k)= (-1)^{\sum |x_i|m_i}S^{-m_1}(x_1\cap \Gamma_{X_1})\otimes \cdots \otimes S^{-m_k}(x_k\cap \Gamma_{X_k})\in S^{-m_1}X_*(M_1)\otimes \cdots \otimes S^{-m_k}C_*(X_k)$. Notice that if $|x_i|$ is the degree of $X_i$  in  $C^{-*}(X_i)$ (making it a degree $-|x_i|$ cochain), then
 each $x_i\cap \Gamma_{X_i}$ is an $m_i+|x_i|$ chain, so that $S^{-m_i}(x_i\cap \Gamma_{X_i})$ lives in $(S^{-m_i}C_{*}(X_i))_{|x_i|}$ as desired.   Applying $\bar \varepsilon$, this gets taken to $S^{-\sum m_i}((x_1\cap \Gamma_{X_1})\times \cdots \times (x_k\cap \Gamma_{X_k}))$ times $-1$ to the power

$$\sum_i |x_i|m_i+ \sum_{i\geq 2}(m_i\sum_{j<i}(|x_j|+m_j))+e_2(m_1,\ldots,m_k).$$
The first term of this power is carried over from the Poincar\'e duality maps, the second comes from $\Theta$, and the last is the 2nd symmetric polynomial on $m_1,\ldots, m_k$ from the definition of $\bar \varepsilon$. Note that we can consider the second sum to be over all $i$ by defining the null sum $\sum_{j<1}$ to be $0$.

Pulling out the $\Gamma$s gives  $S^{-\sum m_i}(x_1\times \cdots \times x_k)\cap (\Gamma_{X_1}\times \cdots \times \Gamma_{X_k})=S^{-\sum m_i}(x_1\times \cdots \times x_k)\cap \Gamma_{X_1\times \cdots \times X_k}$ at the cost of an additional factor of $-1$ to the 
$$\sum_{l<k} m_l(\sum_{a>l}|x_a|), $$
by \cite[VII.12.17]{Dold}. This gives a total sign of $-1$ to the 

\begin{equation}\label{E: signe}\sum_i |x_i|m_i+ \sum_{i}(m_i\sum_{j<i}(|x_j|+m_j))+e_2+\sum_{i} m_i(\sum_{a>i}|x_a|). 
\end{equation}

To simplify this, notice that for each fixed $m_i$, the terms involving $m_i$ and $|x|$'s are 
$$m_i|x_i|+m_i\sum_{j<i}|x_j|+\sum_{a>i}|x_a|=m_i(\sum_j |x_j|).$$
Summing over $i$ gives all of the terms of \eqref{E: signe} that involve an $|x|$ factor. Looking at the terms that involve only $m$'s, we have $\sum_i m_i\sum_{j<i}m_j+e_2\equiv 0\mod 2$, since these terms are identical. Thus the sign is $(\sum_i m_i)(\sum_j |x_j|)$.

On the other hand, the top map of the diagram simply takes $x_1\otimes \cdots \otimes x_k$ to $x_1\times \cdots \times x_k$, while the righthand map takes this to $S^{-\sum m_i}(x_1\times \cdots \times x_k)\cap \Gamma_{X_1\times \cdots \times X_k}$ with a sign of $(-1)$ to the 
$$(\sum_i |x_i|)(\sum_j m_j).$$ This completes the proof.

\end{proof}

The next lemma demonstrates that our umkehr map $\Delta_!$ (see Section \ref{S: transfer}) is a chain map of the appropriate degree. The same proof works for a PL map between manifolds $f:X^n\to Y^m$ and the ensuing transfer $f_!$ defined on the complex $C^f_*(Y)$ of chains in general position with respect to $f$; see \cite{McC}. We state and prove the lemma for both cases at once. For the case of $\Delta$, we take $Y=X(k)$. 

\begin{lemma}\label{L: transfer chain}
Suppose $f:X^m\to Y^n$ is either a PL map of manifolds or $f=\Delta:X\to X(k)$. Suppose $C\in C^f_i(Y)$, as defined for manifolds in \cite{McC} or for $f=\Delta$ as defined in Definition \ref{D: GPC}. Then $f_!\bd C=\bd f_! C$. Thus $f_!$ is a degree $0$ chain map. 
\end{lemma}

\begin{proof}
The lemma is a consequence of the following diagram. 
\begin{diagram}
S^{-m}H_i(|C|\cup\Sigma_Y,|\bd C|\cup \Sigma_Y) &\rTo^{\bd_*}& S^{-m}H_{i-1}(|\bd C|\cup \Sigma_Y,\Sigma Y)\\
\uTo^{(-1)^{(m-i)m}(\cdot\cap \Gamma_Y)}_\cong &&\uTo_{(-1)^{(m-i+1)m}(\cdot\cap \Gamma_Y)}^\cong\\ 
H^{m-i}(Y-|\bd C|\cup\Sigma_Y,Y-|C|\cup\Sigma_Y) &\rTo^{\delta^*} & H^{m-i+1}(Y-\Sigma_Y, Y-|\bd C|\cup\Sigma_Y)\\
\dTo^{f^*}&&\dTo_{f^*}\\
H^{m-i}(X-|\bd C|'\cup\Sigma_X,X-|C|'\cup \Sigma_X)&\rTo^{\delta^*}&H^{m-i+1}(X-\Sigma_X,X-|\bd C|'\cup \Sigma_X)\\
\dTo_\cong^{(-1)^{(m-i)n}(\cdot \cap \Gamma_X)} &&\dTo^\cong_{(-1)^{(m-i+1)n}(\cdot \cap \Gamma_X)}\\
S^{-n}H_{n-m+i}(|C|'\cup \Sigma_X,|\bd C|'\cup\Sigma_X)&\rTo^{\bd_*}& S^{-n}H_{n-m-1}(|\bd C|'\cup\Sigma_X, \Sigma_X)\\
\uTo^\cong&&\uTo_\cong\\
S^{-n}H_{n-m+i}(|C|',|\bd C|')&\rTo^{\bd *} &S^{-n}H_{n-m-i}(|\bd C|').
\end{diagram}
The second and forth square commute by the naturality of $\delta^*$ and  $\bd_*$. Using the formula, $\bd S^{-m}(d^p\cap c_{p+q})=(-1)^m(\delta d^p\cap c_{p+q}+(-1)^{p}d^p\cap \bd c_{p+q})$ (see \cite[page 243]{Dold} and recall that shifting by $m$ adds a sign of $(-1)^m$ to the boundary map), the first square $(-1)^{m+(m-i)m+(m-i+1)m}=+1$ commutes, while, similarly, the third square  commutes. 
Thus, overall, the outer rectangle commutes. 

Now, the composition along the sides of this diagram represent $f_!$, so the proof is completed by observing that the association of homology classes with chains is also natural - see \cite[Lemma 4.1]{McC} or \cite[Section 1.2]{GM1}.
\end{proof}

The following lemma corrects the commutativity of Lemma 10.5b from the original version of \cite{McC}. We do not need this lemma directly in this paper, though the special case where all maps are generalized diagonals seems to be implicit in the proofs of Section \ref{S: GM}.
 Here we leave the lemma stated for manifolds rather than define the appropriate general position and stratified map notions for pseudomanifolds.

\begin{lemma}\label{L: mccom}
Given PL maps of manifolds $f_i: X^{n_i}_i\to Y_i^{m_i}$, the following diagram commutes:

\begin{diagram}
S^{-m_1}C^{f_1}_*(Y_1)\otimes \cdots \otimes  
S^{-m_k}C^{f_k}_*(Y_k) &\rTo^{\bar{\varepsilon}_k}
&
S^{-\sum m_i}C_*^{f_1\times \cdots \times  f_k}(Y_k\times \cdots \times  Y_k)\\
\dTo^{f_{1!}\otimes\cdots\otimes  f_{k!}} && \dTo_{(f_1\times\cdots\times  f_k)_!}\\
S^{-n_1}C_*(X_1)\otimes \cdots \otimes   S^{-n_k}C_*(X_k) &\rTo^{\bar\varepsilon_k}
& S^{-\sum n_i }C_*^{f_1\times \cdots \times  f_k}(X_1\times \cdots \times X_k).
\end{diagram}
\end{lemma}

\begin{proof}
Let $S^{-m_1}x_1\otimes \cdots \otimes  S^{-m_k}x_k$ be a generator of $S^{-m_1}C^{f_1}_*(Y_1)\otimes \cdots \otimes  
S^{-m_k}C^{f_k}_*(Y_k) $. 
The lefthand vertical map takes this to $S^{-n_1}\chi_1\otimes \cdots \otimes  S^{-n_k}\chi_k$, where $\chi_i$  is the chain represented by the Poincar\'e dual in $X_i$  of the pullback by $f_i^*$ of the Poincar\'e dual in $Y_i$  of $x_i$ (see the definition of $\Delta_!$ in Section \ref{S: transfer} and of $f_!$ in \cite{McC}). Here we take the Poincar\'e duals with the appropriate signs as discussed in that section. If $x_i$ has degree $|x_i|$, then $\chi_i$ has degree $n_i-m_i+|x_i|$.

By definition of $\bar \varepsilon$, the bottom map takes $S^{-n_1}\chi_1\otimes \cdots \otimes  S^{-n_k}\chi_k$ to $S^{-\sum n_i}(\chi_1\times \cdots \times \chi_k)$ times $-1$ to the 
$$e_2(n_1,\ldots,n_k)+\sum_i  n_i(\sum_{j<i} (n_j-m_j+|x_j|)). $$
 Here $\chi_1 \times  \cdots \times \chi_k$ is the cross product of chains.

On the other hand, $\bar \varepsilon(S^{-m_1}x_1\otimes \cdots \otimes  S^{-m_k}x_k)$ equals $S^{-\sum m_i} x_1\times \cdots \times x_k$ times $-1$ to the 
$$e_2(m_1,\ldots, m_k)+\sum_i  m_i(\sum_{j<i} |x_j|)  .$$ 

The righthand map then applies the transfer $(f_1\times \cdots \times f_k)_!$ to this.

To resolve these signs, we must compare  how the Poincar\'e duality maps  on products compare to the Poincar\'e duals in the individual spaces. In particular, looking only at the signs that arise within the transfer (and ignoring for the moment those that have already come into the formulas above from the definition of $\bar \varepsilon$ and from the shift isomorphisms), we have that each $\chi_i= (-1)^{m_i(m_i-|x_i|)+n_i(m_i-|x_i|)}(f_i^*( x_i\Upsilon_{Y_i}))\cap \Gamma_{X_i}$. Here we recall that $\Upsilon$ simply refers to the inverse of the cap product (recall that our chains are actually represented by homology classes; see Section \ref{S: transfer}), the first summand in the power of $-1$ comes from the Poincar\'e duality map associated with $\Upsilon$, and the second summand in the power of $-1$ comes from the Poincar\'e duality map associated with  $\cap \Gamma_{X_i}$. 

Thus \begin{align*}
\chi_1&\times \cdots \times \chi_k =
(-1)^{\sum_i (m_i(m_i-|x_i|)+n_i(m_i-|x_i|))} ((f_1^*( x_1\Upsilon_{Y_1}))\cap \Gamma_{X_1} )\times \cdots \times ((f_1^*( x_k\Upsilon_{Y_k}))\cap \Gamma_{X_k} )\\
&=(-1)^{\sum_i (m_i(m_i-|x_i|)+n_i(m_i-|x_i|)) +\sum_i n_i(\sum_{j>i}(m_j-|x_j|)}  (f_1^*( x_1\Upsilon_{Y_1})\times \cdots \times f_k^*( x_k\Upsilon_{Y_k}))   \cap (\Gamma_{X_1}\times \cdots \times \Gamma_{X_k})\\
&\qquad \text{from pulling out the cap products; see \cite{Dold}}\\
&=(-1)^{\sum_i (m_i(m_i-|x_i|)+n_i(m_i-|x_i|)) +\sum_i n_i(\sum_{j>i}(m_j-|x_j|)} 
(f_1\times \cdots \times f_k)^*( x_1\Upsilon_{Y_1}\times \cdots \times  x_k\Upsilon_{Y_k})    \cap \Gamma_{X_1\times \cdots \times  X_k}\\
&=(-1)^{\sum_i (m_i(m_i-|x_i|)+n_i(m_i-|x_i|)) +\sum_i n_i(\sum_{j>i}(m_j-|x_j|) +\sum_i m_i(\sum_{j>i}(m_j-|x_j|) } \\
&\qquad\qquad\qquad\qquad\qquad\qquad\qquad\qquad\qquad
(f_1\times \cdots \times f_k)^*( (x_1\times \cdots \times x_k) \Upsilon_{Y_1\times \cdots Y_k}   ) \cap \Gamma_{X_1\times \cdots \times X_k}
\end{align*}

Finally, the power of $-1$ arising from the two Poincar\'e duality maps in the definition of $(f_{1}\times \cdots \times f_{k})_!$ is $$(\sum_i m_i)(\sum_j m_j-\sum_j |x_j|) + (\sum_i n_i)(\sum_j m_j-\sum_j |x_j| )  .$$

Altogether, we now have four sets of signs that we need to have cancel out:
\begin{gather*}
e_2(n_1,\ldots,n_k)+\sum_i  n_i(\sum_{j<i} (n_j-m_j+|x_j|))\\
e_2(m_1,\ldots, m_k)+\sum_i  m_i(\sum_{j<i} |x_j|)\\
\sum_i (m_i(m_i-|x_i|)+n_i(m_i-|x_i|)) +\sum_i n_i(\sum_{j>i}(m_j-|x_j|) +\sum_i m_i(\sum_{j>i}(m_j-|x_j|)\\
(\sum_i m_i)(\sum_j m_j-\sum_j |x_j|) + (\sum_i n_i)(\sum_j m_j-\sum_j |x_j| )
\end{gather*}

To see that these  powers of $-1$ indeed do cancel each other (for which we only need to work mod $2$), first observe that for each fixed $n_i$ if we only look at terms involving $n_i$ and the various $x_j$, the first expression gives us $n_i\sum_{j<i}|x_j|$, the second gives no such term, the third gives $n_i(|x_i|+\sum_{j>i}|x_k|)$, and the last provides $n_i\sum_j|x_j|$, so these all cancel. Similarly, looking at terms involving only $m_i$ and some $|x_j|$, the first expression provides none of these, the second provides $m_i\sum_{j<i}|x_j|$, the third provides $m_i(|x_i|+\sum_{j>i}|x_j|)$, and the last provides $m_i\sum_j |x_j|$, so these all cancel. For terms involving just the various $n_i$, the first equation has $e_2(n_1,\cdots, n_k)$ and $\sum_i\sum_{j<i}n_in_j$, which cancel. For terms involving just $m_i$s, the second expression gives $e_2(m_1,\ldots,m_k)$, the third has $\sum_i m_i(m_i+\sum_{j>i}m_j)$, and the last expression has $\sum_i \sum_j m_im_j$. To see that these all cancel out, notice that $e_2(m_1,\ldots,m_k)=\sum_i\sum_{j>i}m_j$, that all of the cross terms in $\sum_i \sum_j m_im_j$ are repeated and are thus $0\mod 2$, and that the remaining terms in 
$\sum_i \sum_j m_im_j$ are precisely $\sum_i m_im_i$. Finally, we examine terms of the form $m_in_j$. For each fixed $n_i$, the first expression contributes $n_i\sum_{j<i}m_j$, the third contributes $n_im_i$ and $n_i\sum_{j>i}m_j$, and the last contributes $n_i\sum_j m_j$, which all cancel.
\end{proof}

Several diagrams in this paper were typeset using the \TeX\, commutative
diagrams package by Paul Taylor.

\bibliographystyle{amsplain}
\bibliography{bib}

\end{document}